\documentclass[11pt]{amsart}

\usepackage[section]{placeins}   

\usepackage{amsmath}
\usepackage{amsthm}
\usepackage{amssymb}
\usepackage[T1]{fontenc}
\usepackage{ae,aecompl}
\usepackage[all, knot, poly]{xy}
\input xy
\xyoption{all}

\newtheorem{thm}{Theorem}[section] 
\newtheorem*{thm*}{Theorem} 
\newtheorem{cor}[thm]{Corollary} 
\newtheorem{lem}[thm]{Lemma} 
\newtheorem{conj}[thm]{Conjecture} 
\newtheorem*{conj*}{Conjecture} 
\newtheorem{prop}[thm]{Proposition} 
\newtheorem{defn}[thm]{Definition} 
 
\newtheorem{exa}[thm]{Example}

\newtheorem*{ques*}{Question}

\def\all{\ \forall}
\def\inv{^{-1}}

\def\R{\mathbb{R}}
\def\Z{\mathbb{Z}}

\def\G{\mathcal{G}}

\def\st{\mbox{\ s.t.\ }}

\def\comp{\circ}
\def\set#1{\left\{ #1\right\}}
\def\ie{\textit{i.e.\ }}
\def\eg{\textit{e.g.\ }}

\def\inclusion#1#2{\xymatrix{ #1\ar@{^{(}->}[r] & #2}}
\def\incl{\ar@{^{(}->}}
\def\maps{\ar@{|->}}
\def\monic{\ar@{>->}}

\def\bd{\begin{defn}}
\def\ed{\end{defn}}
\def\ec{\end{cor}}
\def\bc{\begin{cor}}
\def\bl{\begin{lem}}
\def\el{\end{lem}}
\def\bt{\begin{thm}}
\def\et{\end{thm}}
\def\bex{\begin{exa}}
\def\eex{\end{exa}}
\def\bp{\begin{prop}}
\def\ep{\end{prop}}
\def\ben{\begin{enumerate}}
\def\een{\end{enumerate}}
\def\bi{\begin{itemize}}
\def\ei{\end{itemize}}
\def\be{\begin{equation}}
\def\bpm{\begin{pmatrix}}
\def\epm{\end{pmatrix}}
\def\ee{\end{equation}}

\def\cofun#1#2{\xymatrix{ #1 \ar@{ {~}{>}}[r] &  #2}  }
\def\confun#1#2{\xymatrix{ #1&  #2 \ar@{ {~}{>}}[l] }  }

\def\a{\alpha}

\def\nullset{\emptyset}

\usepackage{url, enumitem, amscd}
\usepackage{geometry}
\usepackage{graphicx}
\graphicspath{{Pictures/}}
\usepackage[usenames,dvipsnames]{color}
\usepackage[T1]{fontenc}
\usepackage{lmodern}

\def\S{{\mathcal S}}

\def\I{{\mathcal I}}
\def\K{{\mathcal K}}

\def\G{{\mathcal G}}

\def\D{{\mathcal D}}

\def\MM{{\mathfrak M}}

\def\I{{\mathcal I}}

\def\S{{\mathcal S}}

\DeclareMathOperator*{\HFK}{HFK}

\DeclareMathOperator*{\HF}{HF}

\title[Grid Movies]{Grid Movies}
\author[Matthew Graham]{Matthew  Graham}
\address{Department of Mathematics, 
Northwestern University,
2033 Sheridan Road Evanston,
Il 60208-2730 }
\email{mdgraham@math.northwestern.edu}

\begin{document}

\begin{abstract}
We present a grid diagram analogue of  Carter, Rieger and Saito's smooth movie theorem.  Specifically, we give definitions for grid movies, grid movie isotopies and  present a definition of grid planar isotopy as a particular subset of the grid diagram moves: stabilization, destabilization and commutation.  We show that grid planar isotopy classes are in 1-1 correspondence with smooth planar isotopy classes by using a new planar grid algorithm that takes a smooth knot diagram to a grid diagram. We then present generalizations of both the smooth and grid movie theorems that apply to surfaces  with boundary.

\end{abstract}

\maketitle

\section{Introduction}

The motivation for this topological work came from knot Floer homology.  We briefly sketch the chain of ideas leading up to the questions addressed in this paper presently.

In 2000 Ozsv\'{a}th and Szab\'{o} \cite{OS2004} introduced Heegaard Floer homology ($\HF$).  
$\HF$ is a very powerful theory that provides a unified approach to studying 3-and 4-manifolds and the knots and links within them.  
It comes in several flavors: the hat version $\widehat{\HF}$ is the simplest; the minus and plus versions $\HF^-,\HF^+$ contain more geometric information; and there are other versions.  
Knot Floer homology ($\HFK$), the version used to study knots and links was defined by Ozsv\'{a}th and Szab\'{o} \cite{OS:Knot2003}  and independently Rasmussen \cite{Rasmussen2003}.  
It uses  Heegaard diagrams and holomorphic disks and is an invariant of  oriented links in the 3-sphere, which, among other things,  detects the genus of a link $L$, determines if $L$ is fibered, and categorifies the Alexander polynomial.  
That is, the graded Euler characteristic of $\HFK(L)$ is the Alexander polynomial of $L$.  
A combinatorial method of calculating knot Floer homology was presented in \cite{MOS2006, MOT2009}.  
In \cite{MOST2006}, a completely combinatorial construction of $\HFK$, independent of holomorphic techniques, was given.  
These methods use a grid diagram representation of  an oriented knot as input.  
A \emph{grid diagram} of \emph{grid index} $n$ is an $n\times n$ grid where: there is  a single $X$ and a single $O$ in each column and row; vertically the $X$'s are connected to the $O$'s; horizontally the $O$'s are connected to the $X$'s; all vertical lines are overcrossings; the grid has torus identifications.

Recently, Sarkar \cite{Sarkar2010} defined  several maps between the $\HFK$ complexes of two different knots.  
In particular, he defined chain maps,  whose underlying grid diagram maps could be viewed as  births, deaths and saddles of a grid cobordism connecting two grid diagrams.  
Juh\'{a}sz \cite{Juhasz2009} using sutured Floer homology has already shown that $\widehat{\HFK}$  is functorial with respect to smooth decorated cobordisms.  
This leads one to conjecture the following.
\begin{conj}
Grid diagram maps  induce maps on $\HFK^-$ that are invariant with respect to marked smooth isotopy classes of surfaces.
\end{conj}

At this point, this conjecture is premature for at least two reasons: (1) the correspondence between smooth surfaces and these sequences is not understood which also means that it is not understood how smooth surface isotopies relate to the yet undefined notion of grid isotopy; (2) little is known about marked smooth isotopy classes.
The purpose of this paper is to  present the  correspondences and definitions needed for the first point.  That is, we define  grid movies and grid movie isotopies and give the correspondences between, respectively, smooth embedded surfaces in 4-space and smooth isotopies of these surfaces.
We take the first step in  addressing the marking concerns dealing with the second point in   \cite{Graham2014}.

However, before moving on we give a  brief explanation of why the marking correspondences  are necessary.
It was conjectured that  knot Floer homology is only well defined with a particular type of marking. 
In fact Sarkar in \cite{Sarkar2011} has shown that rotating this marking once around a knot induces a non-trivial automorphism of the knot Floer homology chain complex for most of the 85 prime knots up to nine crossings.
Therefore, we eventually need to understand the relation between marked grid movies and smooth marked surfaces embedded in 4-space.  However, we will only deal with the unmarked topological aspects in this paper.

In the rest of this section we present a brief introduction to different approaches used to study surfaces smoothly embedded in 4-space and their smooth isotopy classes and then give an outline of the paper.
However we first note that,  even though the inspiration to study grid movies came from $\HFK$, the results in this paper are purely topological in nature and should provide a framework to ask naturality questions of other theories that use grid diagrams.

The main idea is to use a sequence of grid diagrams to represent the embedded surface and then study the composition of maps   induced on the knot Floer chain complex.  
Fortunately, studying surfaces using a collection of knot diagrams is not a new idea.  
Fox studied surfaces using level sets in \cite{Fox62} and Roseman used knotted surface diagrams in \cite{Roseman1998} and Carter and Saito in \cite{CS1993} combined these two approaches and used a continuous collection of knot diagrams to characterize embedded surfaces in 4-space up to ambient isotopy.
Finally, in \cite{CRS1997}, Carter, Rieger and Saito reduced the amount of data needed to represent (and distinguish the isotopy type of) smooth embedded surfaces in 4-space to a specific finite sequence of classical knot diagrams and moves on these sequences.
It is this `second' movie move theorem, presented in \cite{CRS1997}, that will play the dominant role in the current paper since our goal is to present these surfaces as combinatorial objects in terms of grid diagrams.
To do this we will need to  understand  how the combinatorial structure presented in \cite{CRS1997} coincides with grid diagrams.
Furthermore, since this last theory uses elements of all the previous theories, we will briefly review all of the theories mentioned above presently.

To understand the constructions of these various theories we will begin by reviewing the well known procedure for studying knots in $\R^3$ (or $S^3$) using knot diagrams.  
A \emph{knot} $K=K(S^1)$ is a particular embedding of $S^1$ in an ambient space $K\colon{S^1}\to {\R^3}$.  
A \emph{knot diagram} $D_p(K)$ is a projection  of $K$ to a plane, $p\colon K\to \R^2$ along with some extra data  that keeps track of which strand is the overcrossing at each singular point of the projection.  
To understand knots up to ambient isotopy  the Reidemeister moves are required.  
Specifically, they  account for changes that can occur to the singular set of the projection $p$ in an ambient isotopy of  the embedding space $\R^3$.

Embedded surfaces in 4-space can be studied in a similar way.  
A \emph{knotted surface} $K$ is an embedding of some (possibly disconnected) underlying closed surface $F$ in $\R^4, \ K\colon F\hookrightarrow \R^4$.  
As in knot theory, we will  let $K$ refer either to the embedding $K\colon F \hookrightarrow\R^4$ or the image $K(F)$ of the embedding.
 A \emph{knotted surface diagram},  $\K:=p(K)$, is the image of the  projection of  $K$ to a hyperplane of $\R^4, \ p\colon K\to \R^3$  along  with a depiction of crossing information. Roseman \cite{Roseman1998} provided a set of seven Reidemeister type moves,  now known as Roseman moves, that operate on knotted surface diagrams and characterize ambiently isotopic surfaces.
Carter and Saito  refined this procedure.  
In \cite{CS1993} they studied knotted surface diagrams by choosing a Morse function that separates the critical points of the manifold and the singular sets of the projection.  
From this construction they defined a movie of a knotted surface diagram and supplied the Reidemeister type moves, now known as movie moves, needed to characterize ambiently isotopic surfaces that preserve this extra structure.  
In order to present the isotopy class of a knotted surface by a finite sequence of classical knot diagrams Carter, Rieger, and Saito introduced a second height function in each still and supplied a refined, and more numerous set of movie moves in \cite{CRS1997}.  
Stated briefly, when studying isotopy of surfaces in four dimensional space there are three options: (1) use the knotted surface diagrams and the seven Roseman moves; (2) use movies and the 15 movie moves; or (3) use movies with a height function in each still  and use the refined set of 31 movie moves.  

We will use the `second' movie move theorem that requires movies of knotted surface diagrams with an additional fixed height function in each still  and the 31 movie moves in this paper.  However, in order to understand this theorem it is necessary to understand the original movie move theorem.

\subsection{Original Movie Move Theorem}
Intuitively, one should think of a movie as the collection of inverse images associated to a  Morse function (or even a projection) $f\colon\K\to \R$.  
This gives an ordered collection of stills $\set{f\inv(s)}_{s\in \R}$, which one can analyze sequentially much like watching a movie by increasing or decreasing $s$.  

Carter and Saito used  movies of  knotted surface diagrams to characterize embedded smooth surfaces in 4-space up to ambient isotopy.  
Not surprisingly, they needed to keep track of more data in order to do this.  
Specifically, they needed to understand the critical points of the manifold as well as the critical points of the double and triple point sets of the projection.  
In the next few paragraphs, when defining movies of knotted surface diagrams, we follow \cite{CS1993} closely.

\begin{defn}
  For a closed surface $F$  a smooth map $f\colon F\to \R^3$ is \emph{generic} if it is an immersion except for some isolated points, which are cone points of figure 8's.
\end{defn}
The image of a generic map satisfies the following conditions.  
For all $y\in F$ there exists a neighborhood $N(y)$ such that   $(N(y), f(F)\cap N(y))$ is diffeomorphic to: ($B^3$, intersection of $i$ coordinate planes) where $i=1,2,3$ and $B^3$ contains the origin; or ($B^3$, the cone on a figure 8 curve) where the figure 8 curve is in the boundary of $N(y)$.

The \emph{$j$-tuple set} is $S_j=\set{y\in \R^3;\colon \#f\inv(y)=j}$.  
For $j=2,3$ this is the double point and triple point sets respectively.  
A \emph{branch point} is a point $y\in \R^3$ such that the intersection of  any neighborhood $N(y)$ with $f(F)$ contains a cone on a figure 8.  
Note that the closure of the double point set $\overline{S_2}$ contains the branch points and the triple points (see picture on the right of figure \ref{fig:IntuitivePicture}).

We would like to do more than just denote the double and triple point sets.  
We will need critical point data of these sets, which means we need to understand the manifold structure of $\overline{S_2}$.  
To do this,  note that  $\overline{S_2}$ is  the image of a compact 1-dimensional manifold (non-generically) immersed in $\R^3$.  
To define this curve we will use the configuration space
\[\widetilde{C}_r=\set{(x_1, \ldots, x_r)\colon  x_j\in F, x_j\neq x_k \mbox{ for } j\neq k, f(x_1)=\cdots f(x_r)}.\]
There is a free action of the permutation group, $\Sigma_r$ on this space.  
The \emph{$r$-decker} manifold is the associated $r$-fold covering space 
\[D_r=\widetilde{C}_r\times_{\Sigma_r}\set{1,2, \ldots, r}.\]
The quotient $C_r=\widetilde{C}_r/\Sigma_r$ is the \emph{$r$-tuple manifold}.  There are maps 
\begin{align*}
 & \tilde{f}_r\colon D_r\to F && f_r\colon C_r\to \R^3 \\
  & [(x_1, \ldots, x_r); j] \mapsto x_j  && [x_1, \ldots, x_r]\mapsto f(x_1)
\end{align*}
that make the square (involving the function $f$) in the following diagram commute and satisfy $f_r(C_r)=\cup_{i=r}^3 S_i$ for $r=1,2,3$.
\begin{align}
  \xymatrix{D_r\ar[rr]^q \ar[d]_{\tilde{f}_r}& & C_r \ar[d]^{f_r}& \\ F \incl[r] \ar@/_1pc/[rr]_f & \R^4 \ar[r]^p& \R^3 \ar[r]^{p_1}& \R}
\label{eqn:CommutativeDiagram}
\end{align}
The map $q\colon D_r\to C_r$ is the covering projection.  The image of $f_r$ is the $r$-decker set.

To include the branch points in this picture, we will consider them to be in the double point set even though the preimage of a branch point $y$, is a single point.  
With this understanding,  branch points are precisely the boundary of the double point set that lie in the interior of $F$.  
For consistency, we include $f\inv(y)$ in the double decker manifold and extend the covering $q$ to a branched covering.
\begin{defn}
  For a fixed embedding $F\hookrightarrow \R^4$ and generic map $f\colon F\to \R^3$ (where $f:=p\comp K$ as in diagram \ref{eqn:CommutativeDiagram}) a projection $p_1\colon\R^3\to \R$ is a \emph{generic  Morse function} for the knotting if: 
  \begin{enumerate}
  \item $p_1\comp f_r$ has only non-degenerate critical points for all $r=1,2,3$; 
  \item each critical point is at a distinct critical level of $p_1$. (Define critical points to include triple points and branch points.)
  \end{enumerate}
\end{defn}
Throughout this paper we will always choose to draw the Morse height function in the horizontal direction: see figure \ref{fig:IntuitivePicture} for an example.

\begin{defn}
\label{defn:SmoothMovie}
  A  \emph{smooth movie} of the knotted surface $K$ is a knotted surface diagram $\K$ with a fixed generic Morse function.
\end{defn}

Using Morse theory, each movie (of a knotted surface diagram) gives rise to a knotted surface and each knotted surface gives rise to a movie.
\begin{defn}
\label{defn:ComplimentaryCoordinateSystem}
  A \emph{complementary coordinate system} for an embedding is an orthonormal coordinate system of $\R^4$,  $(v,v_1,v_2,v_3)$ that satisfies:
  \begin{enumerate}
    \item  the projection orthogonal to  $v$,  $p\colon\R^4\to \R^3$ is generic;
    \item  projection  of $p(\R^4)$ onto the vector $v_1$ is a generic Morse function $p_1\colon\R^3\to \R$ for $K$ and $p\comp K$.
  \end{enumerate}
\end{defn}

A complementary coordinate system can be chosen for any knotted surface.  To see this, choose any orthonormal system and if the conditions are not satisfied for $v$ and $v_1$ perturb $v$ while remaining in $v_1^\perp$ to make the projection $p$ generic.  Similarly, if needed, perturb $v_1$ while remaining in $v^\perp$ to get a generic Morse function $p_1$ for $K$ and $p\comp K$.  Finally, choose $v_2\in span(v,v_1)^\perp$ and $v_3\in span(v,v_1,v_2)^\perp$

  \begin{figure}[h]
    \centering
    \includegraphics[width=\textwidth]{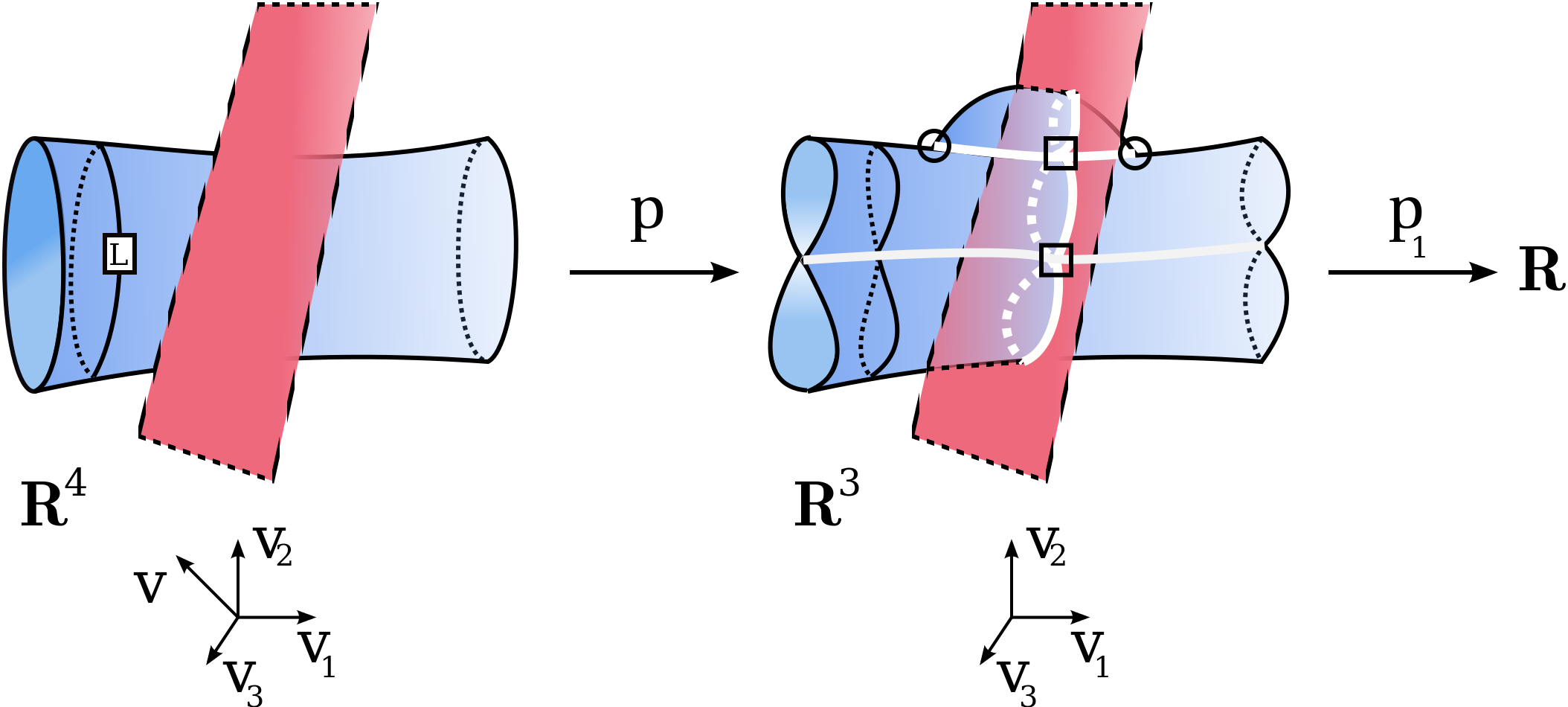}
    \caption[ \ \ Example generic projection and movie.]{Two subsets of a surface $K$ are tracked through a projection $p$ and Morse function $p_1$.  The $L$ appearing on the left indicates the knot type of that particular level set of $K$.  White lines represent the double point set $S_2$, the squares contain triple points and the circles contain branch points.}
    \label{fig:IntuitivePicture}
  \end{figure}

Let us pause for a moment  to give an intuitive picture of how to think about these constructions.  
Let $g\colon \R^4\to \R$ be the projection onto $v_1$ and for simplicity assume $g$ is a Morse function for the surface $F$.
Since  \cite{Fox62} people have studied surfaces in $\R^4$ using level sets.
Denote these level sets of $\R^4$ by $\R^3_t:=g\inv(t)$.  
For all but finitely many values of $t\in \R$, the intersection of the embedded surface $K$ with $\R_t^3$ will be empty or a link.  
In figure \ref{fig:IntuitivePicture} two subsets of a surface $K$ are shown on the left and their images under $p$ are shown on the right. 
The $L$ appearing in one of the subsets indicates the knot type of that particular level set, which in this example is the unknot (since $p(L)\cap S_2$ is a single point).

By glancing at the projection of the simple surfaces appearing in figure \ref{fig:IntuitivePicture} it is easy to see that knotted surface diagrams can get quite complicated, which makes it difficult to keep the relevant structure in the mind.
Since a movie (of a knotted surface diagram) is defined to be the level sets of the projection $p$ it might seem that  we are stuck using  knotted surface diagrams.
However, it turns out that much of the structure can be understood using the level sets of $g$ as long as a complementary coordinate system is chosen.
A complementary coordinate system allows one to  reverse the order of the projections: slice then project rather than project then slice.
For example, in the diagram on the left of figure \ref{fig:IntuitivePicture} consider the level set $\R^3_s$ that  has the link component $L$.  
With a complimentary coordinate system $p(L)=p_1\inv(s)$.
That is to say, the projection $p$ simply takes the link $L$ to a diagram which is the still of the movie.
Said another way, the projections of the level sets of $g$ are diagrams, which happen to be the stills (level sets of $p_1$) corresponding to the movie of $\K$.  Furthermore, since both $g$ and $p_1$ are defined by projection onto $v_1$ the indexing of the level sets is the same (\ie $p\comp g\inv (s)=p_1\inv(s)$).

\begin{defn}
  Two knotted surfaces $f_i\colon F\to \R^4, i=0,1$ are \emph{ambiently isotopic} if there is an isotopy $H\colon \R^4\times [0,1]\to \R^4 \st H(x,0)=x \all x\in \R^4$ and $H(f_0(a),1)=f_1(a) \all a\in F$.
\end{defn}

The main theorem of \cite{CS1993}  is the following. 
\begin{thm}
 (Carter and Saito 1993)
  Two knotted surface movies represent isotopic knottings if and only if they are related by a finite sequence of movie moves 1-7, 23a, 23b, 25-30 (figure 19) or interchanging the levels of distant critical points.
   \label{thm:SmoothMovieMoves}
\end{thm}

The smooth model movie moves 1-7, 23a, 23b, 25-30 (depicted in  figure 19)  can vary crossing information and can be interpreted as sequences starting at the top  or starting at the bottom.  
These moves also represent moves where each still has been reflected horizontally (or vertically or both).

When studying knots using diagrams, the Reidemeister moves are needed to recover ambient isotopy of knots.  
However, the full statement is that the diagrams of any two ambient isotopic knots are related by Reidemeister moves \emph{and} planar isotopies.
Correspondingly, the theorem above leaves the type of isotopy implied.
Two knotted surfaces represented by movies (of knotted surface diagrams) are isotopic if they are related by movie moves, exchanging critical points \emph{and} level preserving isotopies of $\R^3$ (\ie an isotopy through movies).  
This point is crucial to understanding how to derive a grid movie move theorem.

\subsection{The Second Movie Move Theorem.} 

Unfortunately,  the first movie move theorem does not give a combinatorial description of a knotted surface.
  Fortunately, Carter, Rieger and Saito in \cite{CRS1997} provided the needed combinatorial structure with their second movie move theorem.
We will explain these constructions  in the next subsection.  
Presently, we sketch a brief argument that shows why the  structure present in the first movie move theorem is not robust enough to represent smooth movies combinatorially,
which also happens to be one of the motivating factors for the second movie move theorem.

Suppose you are given a smoothly embedded surface in 4-space and then construct a movie of a knotted surface diagram that corresponds to this surface.  
It is a simple matter to  reconstruct a smooth surface that is isotopic to the original since you have the full set of stills from the movie of the knotted surface diagram (even though only a finite sub-collection of stills is used for depiction purposes).
The question we are currently framing is the following, ``Is there a general method of selecting a  finite sub-collection of these movie stills that will always allow one to  reconstruct a smooth surface that is isotopic to the original?  If so, how many stills, and which ones, need to be included in this sub-collection?''

One might initially postulate that only one still in between each of the type I, II, and III critical points is needed since all stills between two successive critical points of the Morse function are planar isotopic to one another.  
Carter and Saito showed that this  sub-collection of stills is not sufficient  to reconstruct a  surface isotopic to the original.  They did this using an example of a surface that had stills between two successive critical points  that, when viewed sequentially in the Morse time direction, looked like a trefoil rotating in a plane.  This example can be found in \cite{CRS1997} (section 3.4.1 and corresponding  figure 15).  
Briefly, it is possible for the link diagram to `rotate'  (a planar isotopy) in between successive critical points, as the time parameter of the Morse function changes, inducing a `twist' in the surface, which can alter the surface isotopy class.
The main point is, given a finite sub-collection of stills with only one still in between critical points of type I, II, and III one would not be guaranteed to reconstruct a surface in the same isotopy class as the original because there is ambiguity in how many times each link diagram is suppose to `twist' in between each successive pair of critical points.

In order to keep track of these `twists' Carter, Rieger and Saito in \cite{CRS1997} introduced  a second fixed height function.
We only briefly describe their constructions and point the interested reader to their paper for more details.
Specifically, given a knotted surface $K$ embedded in 4-space, they: project to an $\R^3$ hyperplane to get a knotted surface diagram ($p\colon \R^4\to \R^3$);  project onto a retinal plane $\pi\colon \R^3\to P$  (an affine plane $P$ that does not intersect the image $p(K)$); lastly they define a height function  $h\colon \R^2\to \R$.
\[\xymatrix{F\incl[r]^K& \R^4 \ar[r]^p& \R^3\ar[r]^\pi & P\ar[r]^h& \R}\] 
They then analyze the singularities, with respect to $h$, of the projection onto the retinal plane.
\begin{defn}
  The inverse images of the $t=a$ and $t=b>a$ levels differ by an elementary string interaction (ESI), with respect to the movie description with a still height function, if only one singular value of the height function $h$ occurs between $a$ and $b$.
\end{defn}
There are seven basic types of ESI's, which  are represented  by the diagrams in the top row and the two diagrams on the bottom left of figure \ref{fig:FESIs}.  Starting from the top left and reading to the right they correspond to: branch points; maximal or minimal critical points occuring on the interior of a double point arc;  isolated triple points (formed by three sheets); birth or death of an unknotted circle (Morse critical point of index 0 or 2); saddle critical point (Morse critical point of index 1); a \emph{cusp} or a \emph{switchback move}; \emph{camel-back move} or a $\psi$ move.

In order to obtain a combinatorial description of a  knotted surface  the `multi-local' moves, which account for exchanging the levels of distant critical points (depicted in the bottom right four diagrams in figure \ref{fig:FESIs}), must be included.

\begin{figure}[h]
  \centering
  \includegraphics[width=\textwidth]{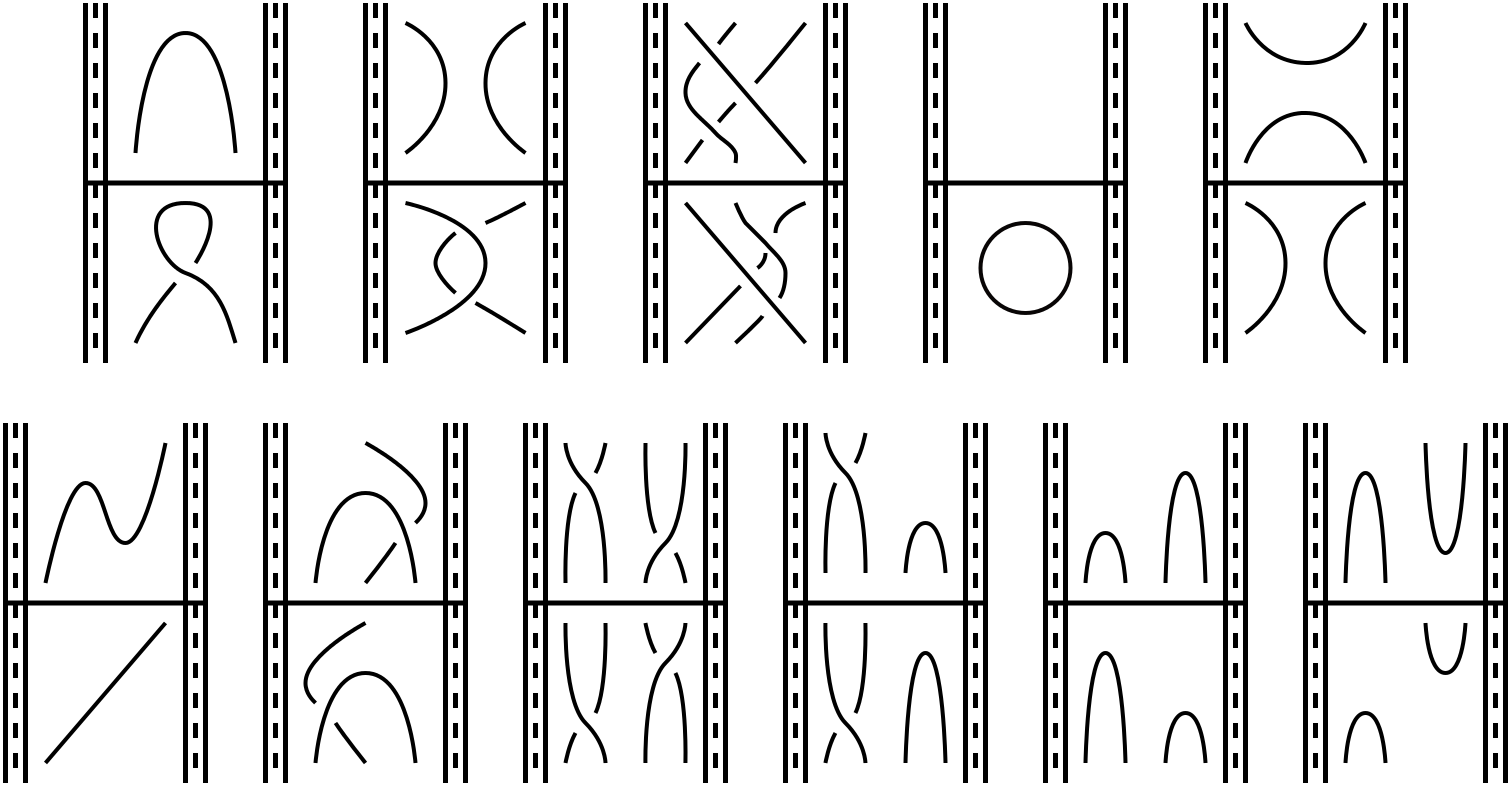}
  \caption{The full set of elementary string interactions (FESI's): the five movie move  ESI's in the top row;  the right four on the bottom row are referred to as multi-local moves.  The left two ESI's on the bottom row are the cusp and $\psi$- moves.}
  \label{fig:FESIs}
\end{figure}

Together the 11  elementary string interaction's in figure \ref{fig:FESIs} make up the full set of elementary string interactions (FESIs).  
Each FESI has several versions obtained by: reflecting each frame about the horizontal (or vertical) axis; changing between positive and negative crossings; orienting the strings differently; and reading the film strip from bottom to top rather than top to bottom.
Taking into account these symmetries, the FESIs represent the stills that need to be included in the finite sub-collection in order to reproduce a smooth surface of the same isotopy class as the original (see theorem 3.5.4 in \cite{CRS1997}).
More explicitly, to generate a sufficient sub-collection of stills to reconstruct a surface of the same isotopy class: using the time parameter, order the stills that contain either a  Morse critical point of type I, II, III or contain a singular point of the height function $h$; and then select a single still in between each of these critical points.  This  ensures that each subsequent still  in this sub-collection differs from the previous one by a FESI (and  possibly a level preserving isotopy of $\R^3$ that respects $h$).

One of the many results in \cite{CRS1997} is the following theorem.
\begin{thm}
\label{thm:CRS1997}
  (Carter, Rieger, Saito 1997) Two smooth movies with height functions on each still represent isotopic knotted surfaces if and only if one can be obtained by the other by a finite sequence of the 31 smooth movie moves with height functions fixed in each still.
\end{thm}

In this paper we prove the following grid diagram analogue of the above theorem.
\begin{thm}
\label{thm:GridMovieMoves}
  Two grid movies represent isotopic knotted surfaces if and only if one can be obtained from  the other by a finite sequence of grid movie moves.
\end{thm}
These moves are depicted in figure \ref{fig:GridMovieMoves}.

In order to  generalize this movie move theorem to a grid diagram analogue a few things need to be understood: (1) how to gridify a movie;   (2) what a grid planar isotopy is (so that we may define a level preserving grid isotopy of a grid movie); and (3) how to represent a movie (of a knotted surface diagram), up to isotopy, by a finite collection of classical knot diagrams.

The second movie move theorem clearly deals with the third point.  The main idea to deal with the first point is to gridify  each   still, while preserving the planar isotopy class, in a combinatorial description of a knotted surface.

Both Cromwell \cite{Cromwell1995} and Dynnikov \cite{Dynnikov2006} have provided algorithms that take a smooth knot diagram to a grid diagram.
In addition they have shown
\begin{thm}
\label{thm:CromwellDynnikov}
  (Cromwell (1995), Dynnikov (2006)) Any two grid diagram presentations of a link $L\subset \S^3$ are related by a finite sequence of commutations, stabilizations and destabilizations.
\end{thm}
We will define these three moves in the next section.
At present, note that this theorem characterizes grid presentations of links up to ambient isotopy of $\R^3$ or $S^3$.
However, it does not characterize grid presentations of links up to planar isotopy, which is needed for the problem at hand.
Although we can not directly use Cromwell and Dynnikov's theorem, it will help us define a correct notion of grid planar isotopy.

The rest of this paper is organized as follows.  
We give a definition of grid planar isotopy.  
We show that grid planar isotopy coincides with the smooth definition.  
That is, the set of equivalence classes of smoothly planar isotopic diagrams is in one to one correspondence with the set of equivalence classes of grid planar isotopic diagrams.  
In order to show this, we  present a different  planar grid algorithm that takes a smooth link and returns a grid link.
We then define grid movies and grid movie moves.
Following this we describe how to obtain a grid movie from a smooth movie with a fixed height function in each still.  In order to do this we will slightly modify our grid planar algorithm to account for a fixed height function.  We will then prove theorem \ref{thm:GridMovieMoves}.   
Finally, since we will be concerned with surfaces with boundary components, we generalize the definitions and constructions to handle this more general setting.

\subsection{Acknowledgments} I thank my advisor Daniel Ruberman for  his patience and mentoring.  I thank  Sucharit Sarkar  for suggesting the  knot Floer homology problem that inspired the current paper. Additionally, I  thank  Adam Levine for helpful conversations.
\section{Grid Planar Isotopy}
\label{sec:GridPlanarIsotopy}

As mentioned in the introduction, both Cromwell \cite{Cromwell1995} and Dynnikov \cite{Dynnikov2006} have given independent proofs of  theorem \ref{thm:CromwellDynnikov}.  Cromwell used Markov moves and Dynnikov used Reidemeister moves.

Theorem \ref{thm:CromwellDynnikov} shows that commutations, stabilizations and destabilizations  of a grid diagram (which will be defined shortly) encode both planar isotopy information and the Reidemeister moves.  
Therefore, since we seek a grid equivalent of planar isotopy, we will remove the types of commutations, stabilizations and destabilizations that encode Reidemeister moves.  
It turns out that we do need to allow a specific combination of two Reidemeister II moves, which we call a transfer move, to fully capture the right notion of grid planar isotopy.  
We will deal with stabilizations and destabilizations together,  enumerate the types of commutations that can occur and then discuss transfer moves.

First, since a grid diagram has torus identifications we will choose a fundamental domain by specifying a vertical arc and horizontal arc, which we label $\a_0$ and $\beta_0$ correspondingly.
The intersection point $\a_0\cap \beta_0$ will always be in the lower left corner of our diagrams.
We index the other vertical arcs sequentially left to right and horizontal arcs from bottom to top.
Columns are  indicated by the left vertical arc of the column and rows are indicated by the lower horizontal arc.
For a grid diagram of index $n$ there are  columns and rows indexed by $0,1, \ldots, n-1$.

A \emph{stabilization} of a grid diagram takes a grid of index $n$ to a grid of index $n+1$ by the following procedure: choose a row (column) and  split it into two adjacent rows (columns); add a new column (row) in the grid diagram; place an $X$ and an $O$ adjacent to one another such that they are in the two modified rows (columns)  and the new column (row).  
A \emph{destabilization} is the opposite of this procedure: delete an $O$ and $X$ that are in adjacent rows (columns); delete the horizontal (vertical) arc that separate them; delete one of the vertical (horizontal) arcs on either side; contract the two rows and two columns.
Examples of stabilizations and destabilizations are shown in figure \ref{fig:StabilizationTypes}. 

The stabilizations and destabilizations either add another crossing (Reidemeister I move) or they don't, in which case they just add a kink. See figure \ref{fig:StabilizationTypes} for examples.  
Define a de/stabilization move to be a \emph{kink de/stabilization} if it does not add an intersection point and define it to be an \emph{R1 de/stabilization} otherwise see figure \ref{fig:StabilizationTypes}.

\begin{figure}[h]
  \centering
 \includegraphics[width=.6\textwidth]{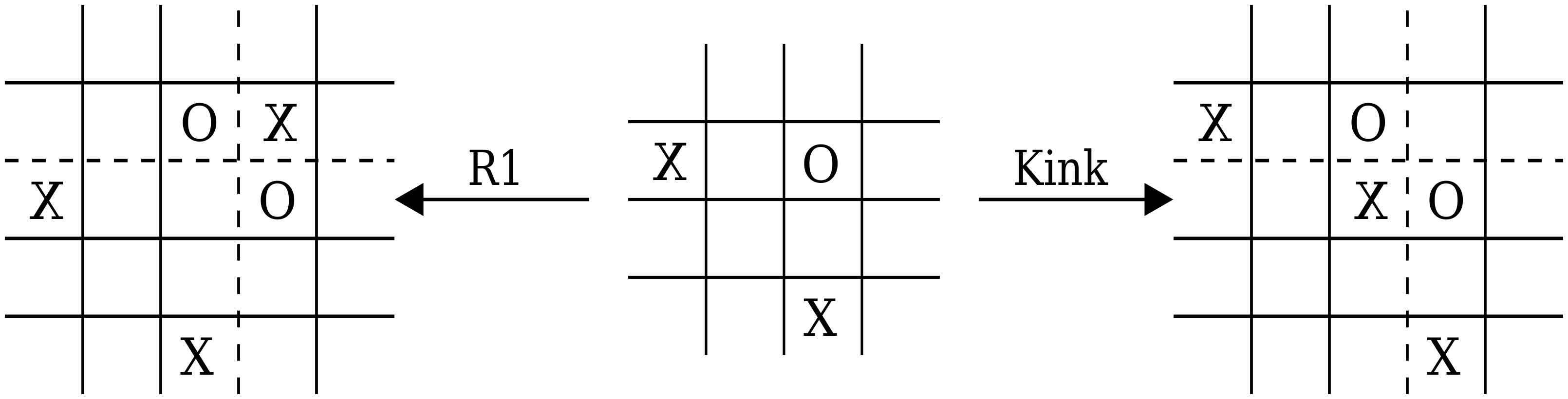}
  \caption[ \ \ R1 and kink de/stabilizations]{An R1-stabilization (left) and a kink stabilization (right). The dotted lines are the lines added to a grid diagram.  Inverting this process yields examples of the two types of destabilizations.
 }
  \label{fig:StabilizationTypes}
\end{figure}

A \emph{commutation} move is an interchange of two adjacent rows (or columns).  
Not all interchanges are allowed.  
In order to determine if a commutation can occur: label the $X$'s and $O$'s by $X_1, O_1, X_2,O_2$ where $X_i,O_i$ are in the same column or row; each of these has row and column coordinates; for a commutation of two columns (rows) order these by the row (column) coordinate; if two share the same row (column) number or if neither $X_1,O_1$ nor $X_2,O_2$ are sequential in this ordering the commutation is obstructed.
An \emph{$\a_i$ commutation} is a commutation between two rows that share the horizontal arc $\a_j$.
Similarly, a \emph{$\beta_j$ commutation} is a commutation between two columns that share the vertical arc $\beta_i$.

There are two types of commutations that can occur:  \emph{un-nested} where both pairs $X_i,O_i$ are sequential; and \emph{nested} where exactly one pair of $X_i,O_i$ are sequential.
Even though $\a_0$ and $\beta_0$ commutations are defined we will not use them since pulling arcs around the back  can represent Reidemeister I moves (see the top left diagram in  figure \ref{fig:CommutationTypes}).
An un-nested commutation is always a grid planar isotopy.  
A nested commutation can represent three things: ($R2$)  a  Reidemeister II move that changes the number of intersections by $\pm 2$; (R3)  Reidemeister III move(s)  in which the separated pair of points (representing an arc of the knot) moves an intersection point  over another intersection point; and ($P$) a grid planar isotopy move  in which the separated pair of points moves an intersection along an arc of the knot.  
See figure \ref{fig:CommutationTypes} for examples.  
A word of caution: many times, when dealing with grid diagrams, the knot is not drawn, which means care is needed when trying to do a grid planar isotopy since the only difference between the  $R3$ move  and a $P$ move  is the existence of the horizontal arc separating both $X,O$ pairs.  
Furthermore, the only difference between an $R2$ move and a $P$ move is the location of one $O$ (or $X$).    Compare the $R2$ and $R3$ moves against the $P$ move in figure \ref{fig:CommutationTypes}.

\begin{figure}[h]
  \centering
  \includegraphics[width=\textwidth]{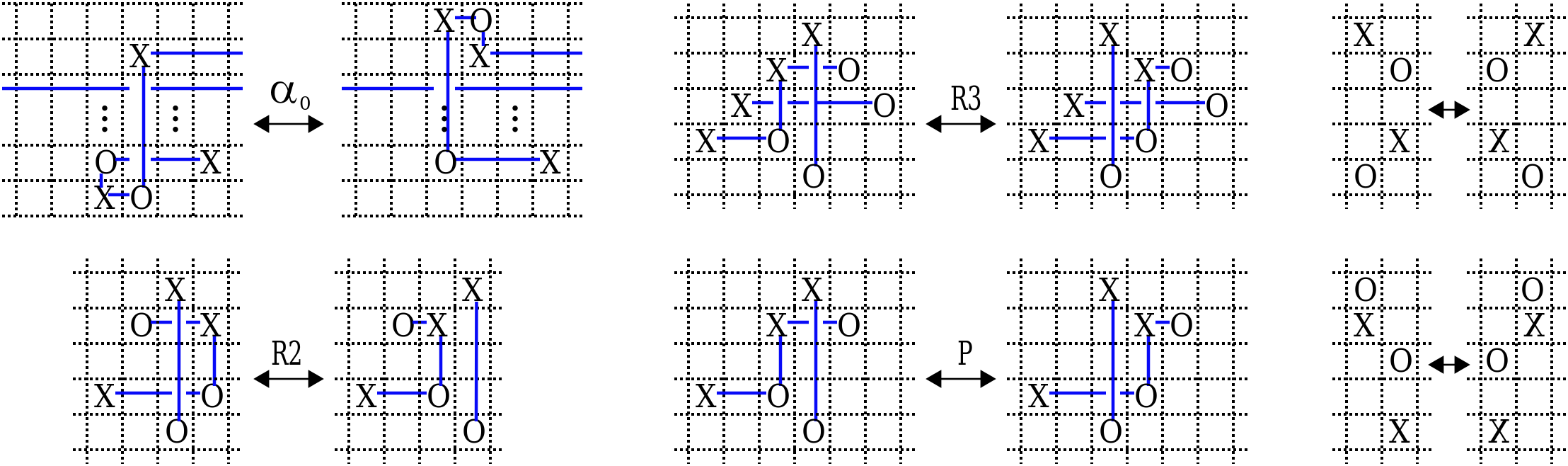}
  \caption[ \ \ Grid Reidemeister moves]{Clockwise from top left: an $\a_0$ commutation that represents an $R1$ move, $R3$ move,  nested pair of arcs,  un-nested pair of arcs,   a $P$ move and an $R2$ move.}
  \label{fig:CommutationTypes}
\end{figure}

Suppose that we defined a grid planar isotopy to be a composition of kink destabilizations, kink stabilizations, un-nested commutations and $P$ moves that are neither $\a_0$ nor $\beta_0$ commutations (all the non-Reidemeister moves).  Then, unfortunately, there would be no grid planar isotopy representing the process of smoothly isotoping an intersection point around a bend in an arc (see figure \ref{fig:TransferMove}).  Notice that the arc that is being isotoped changes its direction in the plane by 18$0^o$.  This together with the fact that commutations and kink de/stabilizations will never reverse the orientations of arcs  means that we cannot represent all planar isotopies with the non-Reidemeister grid moves.  To fix this issue we need to add one more move which is a very specific composition of two Reidemeister II moves called a transfer move depicted in figure \ref{fig:TransferMove}.
\begin{figure}[h]
  \centering
  \includegraphics[width=\textwidth]{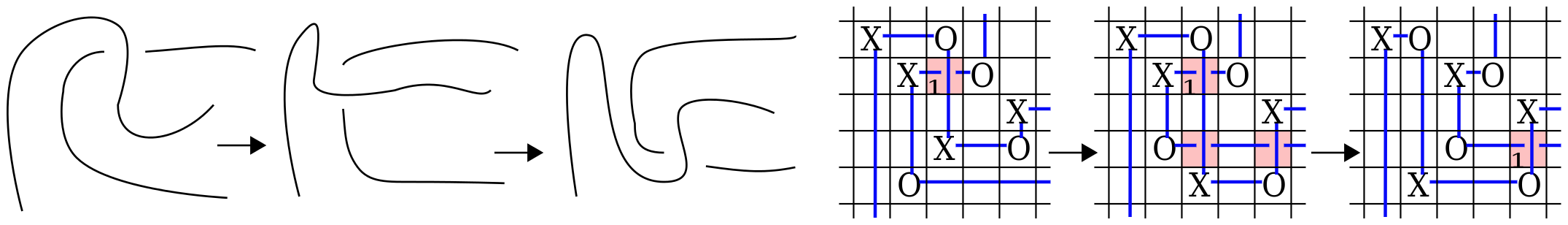}
  \caption[ \ \ Isotoping around corners]{Left: an isotopy that moves the double point smoothly around a corner.  Right: the corresponding transfer move applied to a grid diagram representing the initial smooth diagram.}
  \label{fig:TransferMove}
\end{figure}
\begin{defn}
  A \emph{transfer move} is a composition of two $R2$ moves, one horizontal and one vertical, between two arcs involved in an intersection point $p$.  The first $R2$ move creates two new intersection points of opposite sign between these two arcs.  The second $R2$ move removes $p$ and the newly created intersection point with sign opposite of $p$.
\end{defn}
In more detail, the first $R1$ move either pulls the top strand over the bottom strand or the bottom strand under the top strand, either horizontally or vertically, creating two more intersection points with opposite sign. The second commutation removes the original intersection point and the intersection with the opposite sign as the original, leaving only one of the new intersection points.  If the initial intersection point had an $X$ above and an $O$ below it then the final intersection point will have an $O$ above and an $X$ below (and vice versa).

\begin{defn}
  A \emph{grid planar isotopy} is a finite composition of  kink (de)stabilizations, transfer moves, un-nested commutations, and $P$ moves that are neither $\a_0$ nor $\beta_0$ commutations.
\end{defn}

\section{Planar Grid Algorithm}

The goal of this section is to show that our definition of grid planar isotopy coincides with smooth planar isotopy.  
Specifically,  we present an algorithm for transforming a planar diagram of a knot into a grid diagram and show that it is well defined up to grid planar isotopy.  Then, using the structure that the planar grid algorithm provides, we show that smooth planar isotopy classes are in one to one correspondence with  grid planar isotopy classes.

We begin with some definitions needed for the algorithm and proofs within this section.

\begin{defn}
  A \emph{grid arc} is a closed connected segment of a grid diagram of a knot consisting solely of vertical and horizontal line segments meeting at right angles that  connect two points.
\end{defn}

\begin{defn}
  An \emph{arcword} of a grid arc is some ordering of $L$'s and $R$'s or possibly the single letter $I$.  
Each  $L$ and $R$ represents a left turn and right turn respectively.  
To create an arcword, start at one of the endpoints of the grid arc and travel along the grid  arc until the other end point is reached notating each turn that is passed through.  
If no turns are passed then the arcword is $I$.  
\end{defn}
  In this terminology a kink is either a $LR$ or $RL$ combination.  
Many times $LR$ and $RL$ combinations can be removed or added by a destabilization and stabilization respectively.  

  \begin{defn}
    A \emph{reduced arcword} is an arcword where all possible $LR$ and $RL$ strings have been removed by a grid planar destabilization.  If there are no $L$'s or $R$'s remaining represent this by the arcword $I$.
  \end{defn}
  \begin{defn}
    A  \emph{net turn} is an arcword where all $LR,RL, LLLL,$ and $ RRRR$ terms are algebraically set to the identity $I$.  
  \end{defn}
  For example, the section of the two grid arcs in figure \ref{fig:ArcGridRepresentation} (assuming that the orientation is left to right) would induce a word $I$,  and the word $RLLRRL$ representing the six turns with direction.  
\begin{figure}[h]
  \centering
  \includegraphics[width=.8\textwidth]{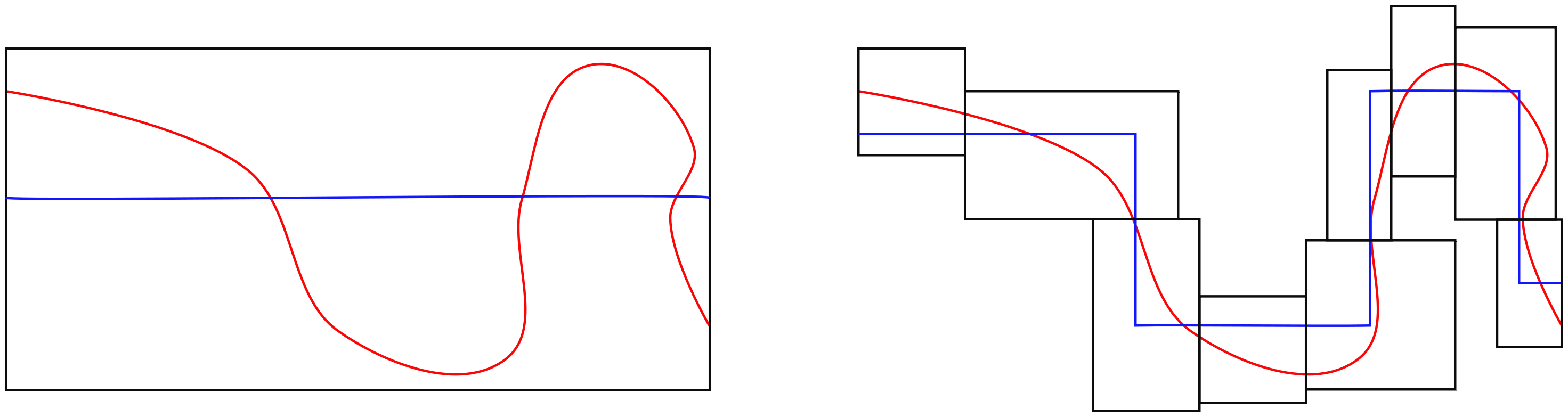}
  \caption[ \ \ Representing grid arcs] {Two possible rectangle choices for the original smooth arc.}
  \label{fig:ArcGridRepresentation}
\end{figure}
Note that both of these arcs have the same reduced arcword of $I$ (and therefore the same net zero turn).  This is true in general: any two sets of rectangles chosen for the same arc will have the same reduced arcword.

\begin{defn}
  An \emph{intersection parallelogram}, for a knot diagram $D$, is a parallelogram $P$ with horizontal sides such that: $P$ contains a single double point and exactly two arcs in its interior; one arc intersects each of the  horizontal sides exactly once, call this the vertical arc; the other arc intersects each of the other two parallel sides exactly once, call this the horizontal arc; the arcs do not intersect the corners of $P$.
\end{defn}

\begin{defn}
  An \emph{intersection square} $R$ is a square that contains a single double point and exactly two arcs in its interior,   an intersection parallelogram $P$, two trapezoids $T_L, T_R$ and four rectangles according to the following rules (see figure \ref{fig:IntersectionSquares} for examples):
  \begin{enumerate}
    \item $\partial R\cap \partial P=\nullset$
    \item The trapezoids $T_L, T_R$ each have at least two right angles; $T_L$ ($T_R$) has a side that is contained in the left (right) non-horizontal side of $P$; this side contains a single intersection point with the horizontal arc contained in $P$.
    \item Each of the four rectangles shares a different corner of $R$.
    \item  The boundary of one rectangle intersects the upper horizontal side of $P$ and the boundary of another rectangle intersects the lower horizontal side of $P$.
    \item The boundary of the other two rectangles intersect  a side of one of the trapezoids.
    \item The boundary of each trapezoid (except for $P$) has exactly two intersection points with a single arc.
  \end{enumerate}
\end{defn}
    \begin{figure}[h]
      \centering
      \includegraphics[width=\textwidth]{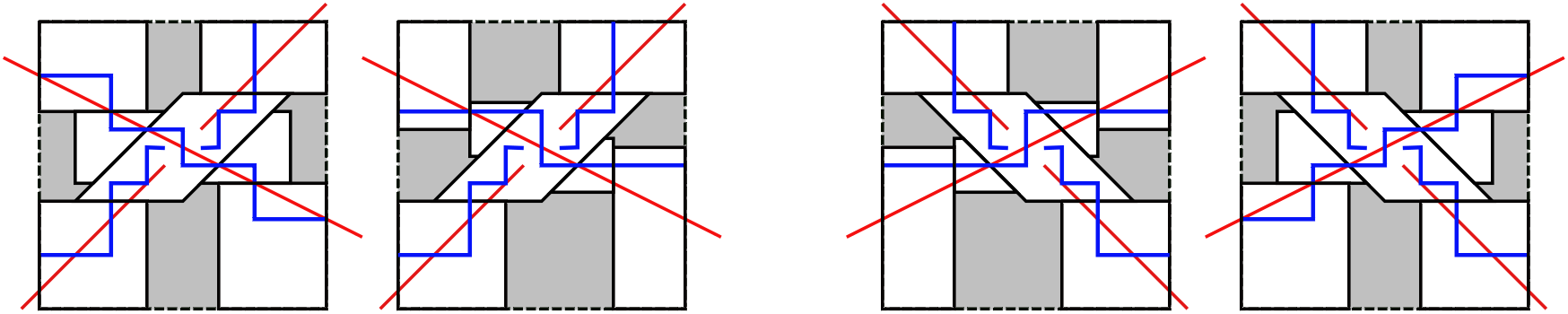}
      \caption[ \ \ Example intersection squares]{Four possible choices of parallelogram and trapezoids involved in an intersection square.  The shaded regions are exterior to all drawn shapes, straight lines are the original arcs, and the arcs with $90^o$ turns  are the grid arcs.}
      \label{fig:IntersectionSquares}
    \end{figure}

\begin{defn}
  The \emph{planar grid algorithm} applied to a diagram $D$ associated to a projection $\pi$ of knot $K$ is enumerated below.
\end{defn}

\begin{enumerate}
  \item Draw an intersection square around each double point.
    \begin{figure}[h]
      \centering
      \includegraphics[width=\textwidth]{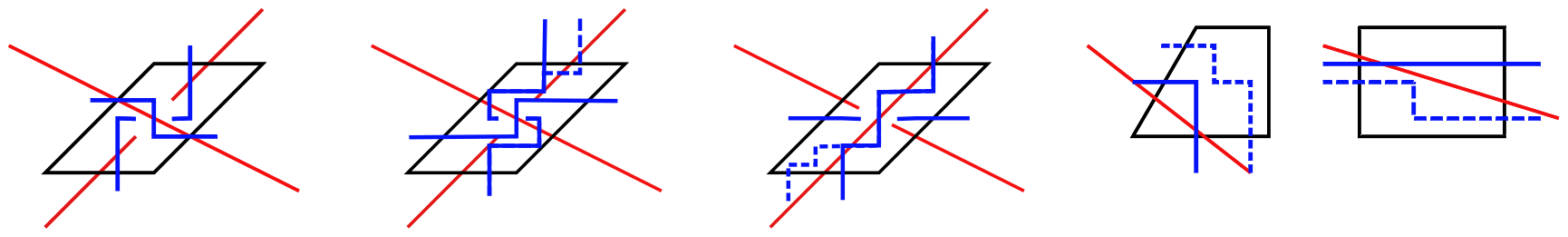}
      \caption[ \ \ Smooth arcs to grid arcs]{Three choices for replacing the original smooth arcs, represented as straight lines,  with grid arcs for an intersection point.  The two choices of travel through the intersection point when the horizontal arc is also the overarc are shown in the left two diagrams.  Dotted lines are other possible grid choices.}
      \label{fig:GridArcReplacement}
    \end{figure}

  \item Number the intersection squares and letter all of the arcs that connect these squares together.  The lettered arcs will have endpoints lying in the boundary of intersection squares and will be disjoint from the interior of all intersection squares.

    \begin{figure}[h]
      \centering
      \includegraphics[width=\textwidth]{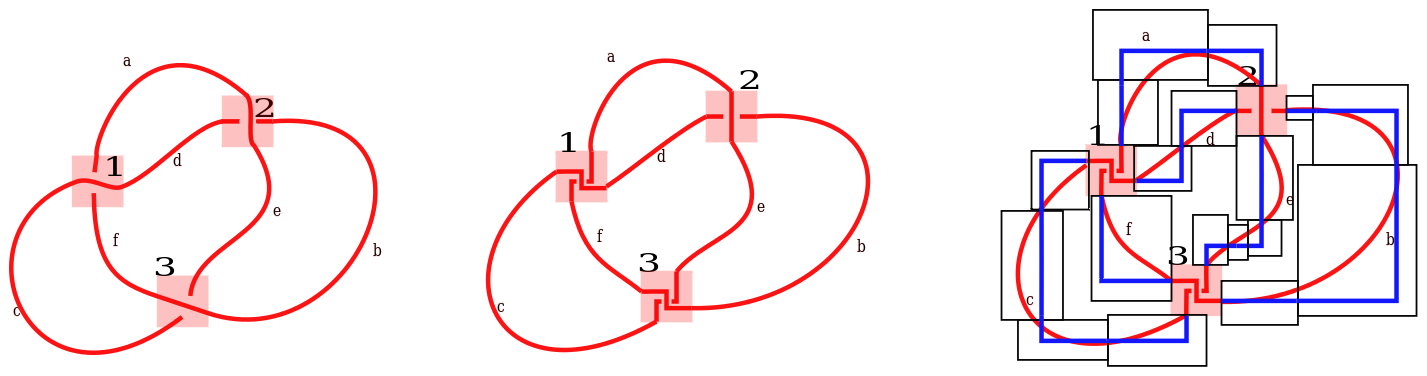}
     
      \includegraphics[width=.65\textwidth]{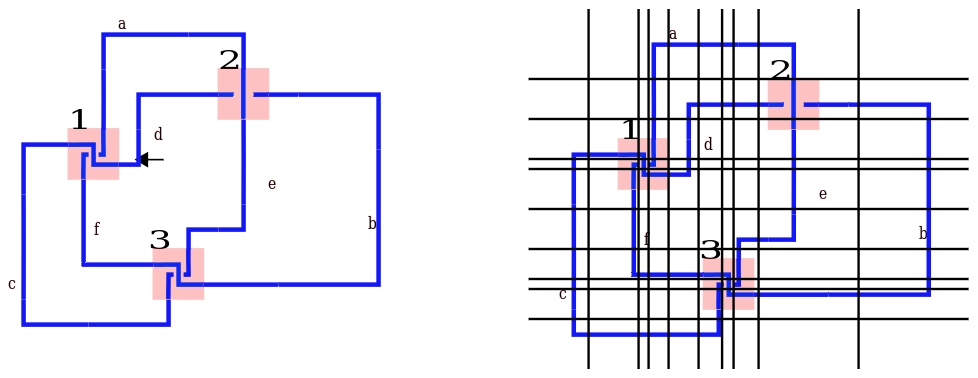}\ \ \ \ \ \ \ \ \ \ \ \includegraphics[width=.25\textwidth]{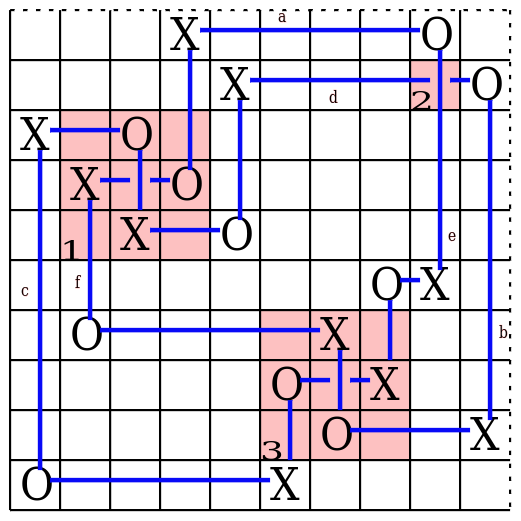}
      \caption[ \ \ Planar grid algorithm applied to trefoil]{The algorithm applied to the trefoil.  The stages are as follows: top left and middle diagrams, steps 1, 2 and 3;  top right, step 4 and 5; bottom left, step 6 (the arrow indicates the vertical line was moved to satisfy non-colinear condition in step 4); bottom middle, step 7; bottom right,  steps 8 and 9.}
      \label{fig:PlanarGridAlgorithm}
    \end{figure}

  \item Approximate each lettered arc by drawing rectangles, with vertical and horizontal sides, such that the arc intersects exactly two distinct edges transversally, each exactly once.
  \item Replace the arcs in the quadrilaterals by grid arcs.  If the quadrilateral is an intersection parallelogram then replace the arcs as shown in figure \ref{fig:GridArcReplacement}.  For trapezoids without double points replace the grid arcs with:
    \begin{enumerate}
    \item a grid arc that has net turn $0^o$ if the arc intersects two edges that are not connected by a vertex;
    \item a grid arc that has  net turn of $\pm90^o$ if the arc intersects two edges that are connected by a vertex.
    \end{enumerate}
  \item Delete all the rectangles used for approximation and the original smooth diagram.
  \item   Make sure that no two straight grid arcs (in the whole diagram) are collinear.  Make small translations of lines as needed to achieve this.
  \item Choose a  minimal set of horizontal and vertical lines that separate each horizontal and vertical grid arc from all others.
  \item Make these separating lines uniformly spaced and starting at one corner place an $X$ then choose a direction to travel along the arc and alternately label the corners $X$ then $O$.  Continue until all corners are labeled.  If the diagram is of a link this process will need to be repeated for each link component. Note that this will induce an orientation on each component of the link. Therefore, if the diagram is already oriented then the initial $X$ placement (of each link component) should be chosen to coincide  with the existing orientation.
  \item Add one more vertical and one more horizontal separating line and make torus identifications to get to a toroidal grid diagram.
\end{enumerate}
\begin{lem}
There exists an intersection parallelogram for every double point of a diagram of a knot.
\label{lem:InteresectionParralellogram}
\end{lem}
\proof
We will represent these arcs by graphs of a function.  
If either arc involved in a double point has an infinite or zero slope then simply draw a rectangle with horizontal and vertical sides to satisfy the conditions of an intersection parallelogram.  

Next,  suppose that neither arc has an infinite or zero slope.
In almost all cases the conditions of an intersection parallelogram can again be satisfied with a rectangle with horizontal sides.  However, a rectangle with horizontal sides will not suffice in general.  For example, the two arcs comprising a double point may have slopes of 1 and -1 in the diagram of the knot.  To handle the general case let $p$ be an intersection point in some projection diagram and linear approximate the two arcs $f_1,f_2$ involved in the crossing by $l_1,l_2$.  Then we have $f_i=l_i+R_i$ where $R_i$ is the remainder, or error term.  If $|f''_i(x)|\leq M_i$ on  some interval $(p-r,p+r)$ then $M_i\frac{r^2}{2}\geq |R_i(x)|$.  Let $M=\max\set{M_1, M_2}$, $H$ the vertical distance between the point $p$ and the horizontal lines, and $L$ the horizontal distance between the other set of parallel lines (see figure \ref{fig:ParallelogramExistenceSetup}). 
\begin{figure}[h]
  \centering
  \includegraphics[width=.5 \textwidth]{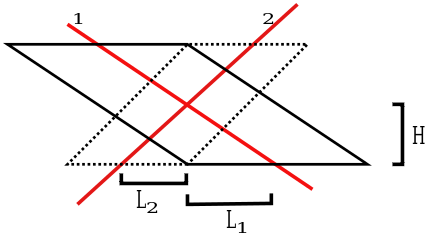}
  \caption[ \ \ Diagram used for existence of intersection parallelograms]{Example intersection parallelogram.}
  \label{fig:ParallelogramExistenceSetup}
\end{figure}
 Suppose, that we chose to make the non-horizontal set of lines parallel to the linear approximation of $f_i$, $l_i$.  Then what we need to show is the following
\[r\geq H\geq M\frac{r^2}{2} \ \ \ \ \ \  \mbox{and} \ \ \ \ \ r\geq L_i\geq M\frac{r^2}{2}.\]  
The left inequalities, of both equations, demand that the parallelogram be inside the region where the inequality for the error term holds and the right inequalities guarantee that the curves have the desired properties (entering and exiting in the desired manner).  Using  $L_i=\frac{H}{B_i}$, where $B_i$ is the slope of the line $l_i$, we see that our inequalities are in terms of $M, B_i, H, r$.  Since the quantities $M$ and $B_i$ are fixed, we only need to choose $r$ and $B$ such that the constraints are satisfied.
Some algebra shows that by letting  $M=c_M10^m$ and $B_i=c_i10^{b_i}$ where $c_M, c_i\in (1,10)$ and choosing $r=10^{(-m-2b_i-2k)}$ and $H=\frac{c_i}{2}10^{(-m-3b_i-3k )}$ where $k$ is a positive integer chosen to make $(-2b_i-k)<0$ satisfies the inequalities.

  Hence there also exists a choice of parallelogram for all non-zero and non-infinite slopes
\qed

\begin{lem}
Let $G(A)$ and $G'(B)$ be any two oriented grid planar arcs without double points satisfying the conditions below,  where $A$ and $B$ are smooth planar arcs without double points and $G(A)$ and $G'(B)$ represent two possibly different applications of the planar grid algorithm on these arcs.  
\begin{enumerate}
 \item  The endpoints coincide in a superposition of the two grid diagrams.
 \item  The initial (resp. final) segments of A and B agree.
 \item  The arcs have the same net turn.
\end{enumerate}
Then $G(A)$ and $G'(B)$ are grid planar isotopic arcs.
\label{lem:GridPlanarIsotopicArcs}
\end{lem}
\proof 

This proof follows essentially the same methods used  to show the PL-Schoenflies theorem (see \cite{Bing83} page 20-21).
Let $G(A)$ and $G(B)$ be given as above.
We need to give a procedure to take the grid arc $G(A)$ to the grid arc $G'(B)$.   
Since these grid arcs are in two different grid diagrams there may be complications in giving a planar grid isotopy taking one grid arc to the other.
Specifically, the other arcs may initially obstruct a desired grid planar isotopy.  
An example of this sort of obstruction is in the first diagram in figure \ref{fig:GridPlanarIsotopicArcs}.  

Initially we assume that  $G(A)$ and $G'(B)$ don't intersect other grid arcs in the superposition.

Let $S$ be the superposition of $G(A)$ and $G'(B)$ such that the endpoints agree.
\begin{figure}[h]
  \centering
  \includegraphics[width=\textwidth]{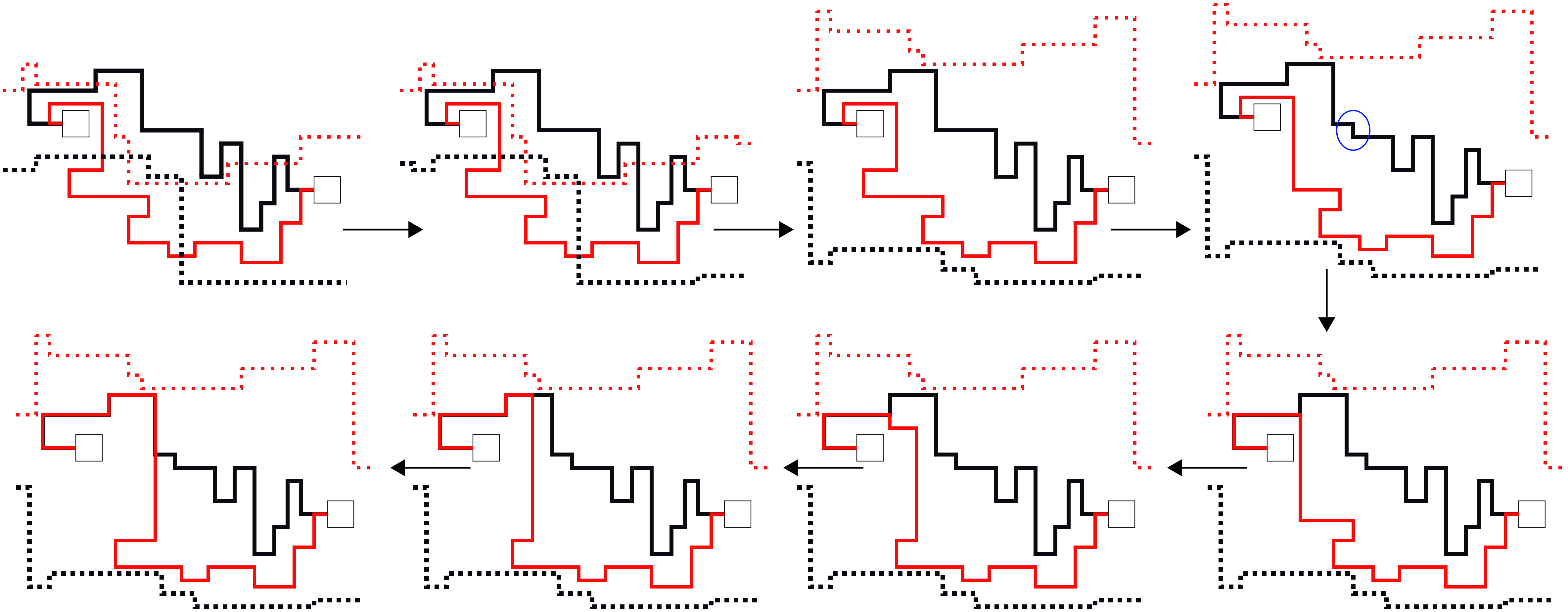}
  \caption[ \ \ Making grid arcs agree]{Two grid diagrams superimposed.  The solid arcs are the two arcs we want to resolve.  The dotted arcs are other arcs that might initially obstruct desired grid planar isotopies.  Note that any apparent double points are an artifact of superimposing the two diagrams.}
  \label{fig:GridPlanarIsotopicArcs}
\end{figure}

Algorithm for resolving $G(A)$ and $G'(B)$ starting from the simple situation depicted in the third diagram on the top row of  figure \ref{fig:GridPlanarIsotopicArcs}:
\begin{enumerate}
   \item Stabilize or destabilize arcs $G(A)$ and $G'(B)$ until they have the exact same number of turns (corners).  This can be done since $net(G(A))=net(G'(B))$.   (The circle in the top right diagram denotes where a stabilization is performed.)  
  \item Perform commutations to move the next turn in each diagram to correspond (bottom right diagram).
  \item If the grid arcs turn in the same direction do nothing.  If they do not, then perform a destabilization followed by a stabilization  that changes the direction of the turn (middle bottom diagrams).  
  \item Repeat steps 3 and 4 until $G(A)=G'(B)$.
\end{enumerate}

Now assume that there are other arcs intersecting $G(A)$ or $G'(B)$ in the superposition.
Note that the condition that neither $G(A)$ nor $G'(B)$ has double points on it means that all of the intersections in the superposition occur between an arc from each still.
Furthermore, this means that the oriented intersection number must be zero between any two arcs.

To planar grid isotope these arcs out of the way, consider two intersection points of opposite sign (connected by an arc with no other double points on it).
Stabilize so that the intersection point satisfies the properties of the lemma and then apply the algorithm just mentioned to make these two sub arcs agree.
Next, apply commutations to one diagram so that the two sub arcs that were in agreement no longer coincide.
Repeat this procedure for all pairs of intersection points.

Hence, we can grid isotope the obstructing arcs out of the way and then apply the procedure above to make $G(A)$ and $G'(B)$ coincide through a planar grid isotopy.
\qed

\begin{prop}
  The output of the planar grid algorithm is well defined.  That is, for a given diagram $D$ two choices of rectangles will lead to grid diagrams of the same grid planar isotopy class.
\label{prop:WellDefinedGridAlgorithm}
\end{prop}
\proof 

Let a diagram $D$ be given and suppose that the numbering of intersections and lettering of the arcs  is the same in both applications $1$ and $2$ of the grid algorithm arriving at grid diagrams $G_1(D)$ and $G_2(D)$.

First, suppose that the intersection squares are chosen the same in both algorithms.  
Then the only possible variation is in the choice of covering rectangles used to get a grid arc for each of the arcs connecting the intersection rectangles.  
Lemma \ref{lem:GridPlanarIsotopicArcs}, applied individually to each of these grid  arcs, states that any differences in $G_1(D)$ and $G_2(D)$ can be rectified by grid planar isotopies.  
There is no difficulty since each resolved grid arc does not obstruct the resolution of any other grid arc.
 Therefore, we only need to consider how different choices of intersection rectangles and parallelograms and trapezoids affect grid diagrams $G_1(D)$ and $G_2(D)$.

First, keep the choice of intersection rectangle fixed and consider the changes to $G_1(D)$ and $G_2(D)$ that result from different choices of intersection parallelograms and trapezoids. There are three choices that must be made: (1) which linear approximation to have parallel to the non-horizontal sides of the parallelogram; (2) deciding whether the horizontal smooth arc is going to intersect the trapezoids $T_{iL}, T_{iR}$ in the vertical or horizontal edges; (3) choice of orientation of the grid arc through the intersection point when the horizontal grid arc is the overarc.  The two left diagrams in figure \ref{fig:GridArcReplacement} give examples of choice 3; choices 1 and 2 are depicted in  figure \ref{fig:IntersectionSquares}.  We will show that all of these choices lead to grid diagrams which are grid planar isotopic.  We will take care of choice 3 and then choices 1 and 2 together.

The grid diagrams resulting from the two  choices of orientation through the intersection point (choice 3) are related by a transfer move, kink stabilizations and kink destabilizations, whose composition is a grid planar isotopy.  The sequence of moves is in figure \ref{fig:OrientationChoiceIsotopy}.
\begin{figure}[h]
  \centering
 \includegraphics[width=\textwidth]{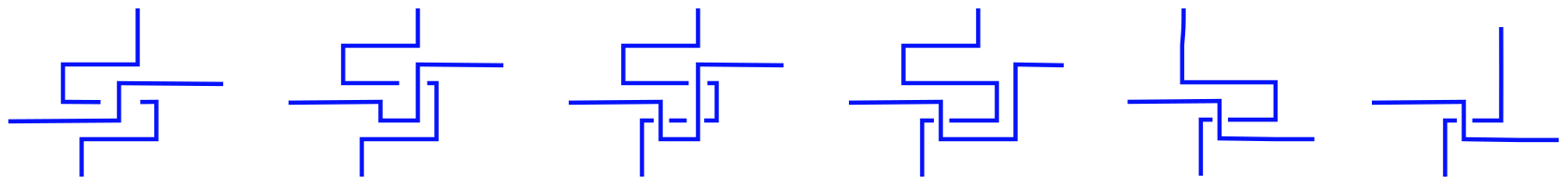}  \caption[ \ \ Reversing orientation of vertical grid arc through a double point]{A sequence of grid planar isotopies taking one choice of grid arcs in an intersection parallelogram to the other choice.  In order from left to right: a stabilization; the transfer move begins with an R2 move;  transfer move ends with an R2 move; destabilizations to the last diagram.}
  \label{fig:OrientationChoiceIsotopy}
\end{figure}

If the orientations through the double point do not  agree then apply the sequence of moves (or the reverse) in figure \ref{fig:OrientationChoiceIsotopy} so that both grid arcs have the same orientation through the double point. Next apply lemma \ref{lem:GridPlanarIsotopicArcs} pairwise to the corresponding eight grid arcs (in two diagrams) with one end point the double point and the other end point being the intersection point of an arc and the  intersection square. This means that the choices involved amounted to a grid planar isotopy.  As an example, figure \ref{fig:IntersectionSquares} depicts the four choices coming from choices 1 and 2.  It can be seen that the four choices of grid arcs do indeed only differ by a sequence of  kink destabilizations which is a grid planar isotopy.

 Hence, we only need to show that different choices of intersection squares also induce differences that are resolved by grid planar isotopies.  To this end, we will show that the difference between two choices of intersection squares where one square sits in the interior of the other square gives rise to differences in $G_1(D)$ and $G_2(D)$ that are resolved by a grid planar isotopy.  Then given two choices of intersection squares say $E$ and $F$ we only need to choose a square that sits inside both of them (or contains both of them) to show that this result holds for all choices of intersection squares. Given two intersection squares $S_1,S_2$ it is clear one can always choose a third intersection square $S$ such that $S_2\supset S\subset S_1$ since double points must reside in the interior of intersection squares.

Suppose that intersection square $E$ is larger than $F$ and contains $F$ in its interior (see figure \ref{fig:IntersectionSquareChoices}).
\begin{figure}[h]
  \centering
  \includegraphics[width=.6 \textwidth]{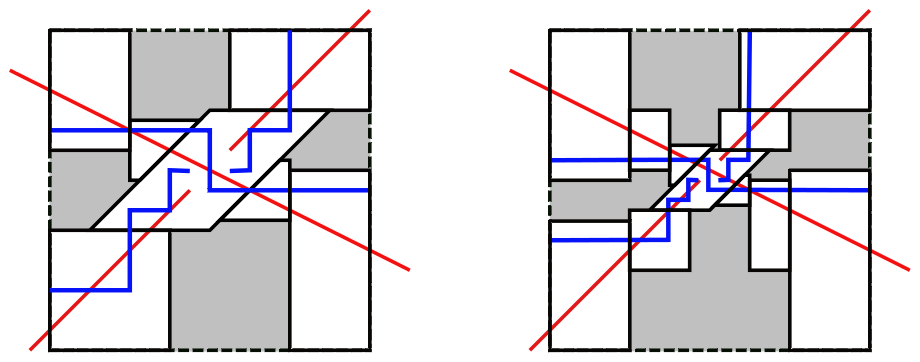}
  \caption[ \ \ Intersection squares used in proof of prop. \ref{prop:WellDefinedGridAlgorithm}] {Large intersection square $E$ on the left and intersection square $F$ on the right.}
  \label{fig:IntersectionSquareChoices}
\end{figure}
  We have shown that choices of shapes on the interior of intersection squares don't matter up to grid planar isotopy.  So make some choices for $E$ and make some choices for $F$.   In the first case $G_1(E)$ we simply apply the normal algorithm to get some grid arcs representing the part of the diagram contained in $E$.  In the second case apply the grid algorithm to the area $F$ and then choose four more rectangles to cover the arcs connecting the boundaries of $F$ and $E$.   By an application of lemma \ref{lem:GridPlanarIsotopicArcs}, each of the four grid arcs in $E$ (with endpoint the double point and some intersection point between the arc and boundary of $E$) is grid planar isotopic to the corresponding grid arc in $F$ adjoin the four covering rectangles.

Hence, the output of the planar grid algorithm is well defined.
\qed

\begin{thm}
Smooth planar isotopy classes are in one to one correspondence with  grid planar isotopy classes.  That is two smooth diagrams $A$ and $B$ are smooth planar isotopic iff  $G(A)$ and $G(B)$ are grid planar isotopic.
\label{thm:IsotopyClasses}
\end{thm}

\begin{proof}
The only if direction is fairly straightforward: smooth the corners and observe that since each grid planar isotopy  is the initial and final state of a smooth planar isotopy then the initial and final states of a composition of finitely many grid planar isotopy will also be the initial and final states of a smooth planar isotopy.  
From the definition of grid planar isotopy it follows that there exists a smooth planar isotopy connecting the beginning and final state.

To see the other direction, let $D, D'$ be  planar isotopic diagrams, let $G(D)$ and $G(D')$ be grid diagrams obtained from applying the planar grid algorithm, and let $h\colon \R^2\times[0,1]\to \R^2$ be an isotopy with $h(D,0)=D$ and $h(D,1)=D'$.  
We need to show that $G(D')$ is grid planar isotopic to $G(D)$.  

There were choices of quadrilaterals, $Q_0$, in getting $G(D)$.  
Fix these choices of quadrilaterals and let $d_1\in [0,1]$ be the first value that the  quadrilaterals do not satisfy the properties set forth in the grid planar algorithm with respect to the arcs in  $h(D,d_1)$.
We will enumerate the ways in which these quadrilaterals can fail shortly.

At $d_1-\epsilon$ choose different quadrilaterals, $Q_1$ that satisfy the properties of the grid planar algorithm with respect to the arcs in  $h(D,t)$  for $t\in [d_1-\epsilon, d_1]$.  Typically, only a few quadrilaterals will be changed.  However, if desired or needed, every single quadrilateral could be chosen differently so that $Q_0\cap Q_1=\nullset$.  The fact that the output of the planar grid algorithm is well defined up to grid planar isotopy (proposition \ref{prop:WellDefinedGridAlgorithm}) means that the different choice of quadrilaterals will lead to grid planar isotopic diagrams.

Repeating the above procedure will yield values $d_i$ for $i=1,\ldots, n$ such that the choice of quadrilaterals $Q_i$ satisfy the needed properties with respect to the arcs in  $h(D,t)$ for $t\in[d_i-\epsilon,d_{i+1}-\epsilon]$.  Let $H_i$ be the grid planar isotopy induced by the different choice of quadrilaterals  at $d_i-\epsilon$.  Suppose that $d_n=1$ and choose $Q_n$ to correspond to the choice of quadrilaterals used to get $G(D')$.  Then $H:  =H_n\comp H_{n-1}\comp \cdots \comp H_1$ is a grid planar isotopy from $G(D)$ to $G(D')$.

\begin{figure}[ht]
  \centering
  \includegraphics[width=\textwidth]{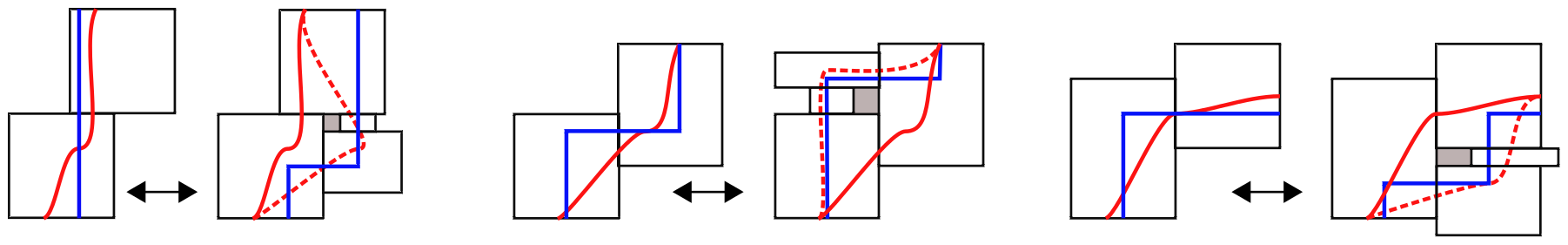}
  \includegraphics[width=\textwidth]{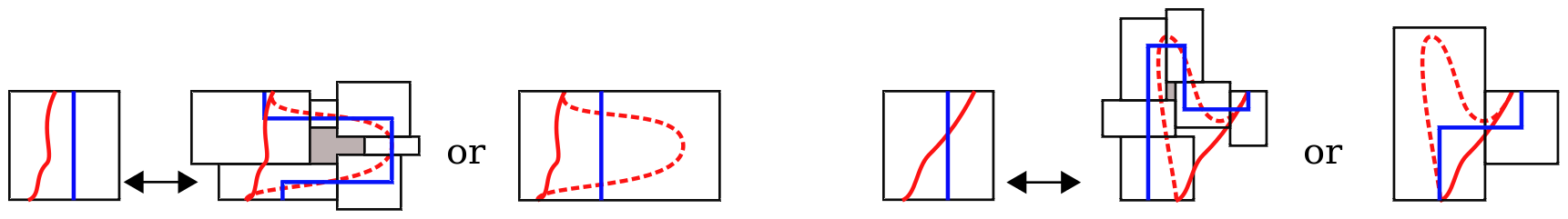}
  \caption[ \ \ Example of failed  quadrilaterals during isotopy] {Five different examples of failing quadrilaterals that are not intersection parallelograms.  The diagrams in the top row isotope a curve past a corner and the diagrams in the second row show two different resolutions for isotoping a curve by a finger move across an edge of a rectangle.  The darker lines are the grid diagram arcs, the lighter solid  line is the original curve and the dotted  lines are the isotoped curve.  Grey shaded regions are exterior to all rectangles.}
  \label{fig:IsotopyRectangles}
\end{figure}
A quadrilateral that fails to satisfy the needed properties required in the grid planar algorithm  either involves an intersection parallelogram or not.  
If it does not involve an intersection parallelogram, then either the smooth arc isotopes around a corner of a trapezoid or it isotopes across a side of a trapezoid (much like a finger move).  
Five different examples  of these two types of  isotopies are shown in  figure \ref{fig:IsotopyRectangles}.  

The top two left diagrams have a reduced arcword of $I$. 
 The top right has a reduced arcword of $R$ or $L$ depending on orientation.  
In general, many more kinks may be  introduced during the  isotopy  and the arcword may  become quite long even though the reduced arcword will not change as can be seen in the diagrams.
Many times a choice of larger trapezoid will account for the isotopy.  A couple of examples of these can be found in the bottom two finger moves in figure \ref{fig:IsotopyRectangles}.

To deal with an isotopy that requires a change in the intersection square, use the same idea used to prove that the algorithm was well defined.  Namely,  any 
time the intersection square is about to fail the conditions set forth in the algorithm (\eg moves across the edge of a parallelogram, etc.) change the intersection square to a small square.  Then choose a larger square containing the small square such that this new intersection square satisfies the conditions of the planar grid algorithm  with respect to the arcs in $h(d,t)$ for $t\in [d_i-\epsilon, d_i]$.  Again these  choices of intersection squares only induce a grid planar isotopy difference between the two diagrams by appealing to proposition \ref{prop:WellDefinedGridAlgorithm}.
\end{proof}

\section{Grid Movies and Grid Movie Moves}
In this section we  prove our main result.  In order to do this we first use the results of the last section to define grid movies and grid movie isotopies.  
We then  discuss: compatibility of the combinatorial structures coming from movies and grid diagrams;  how smooth movies and smooth movie isotopies are represented by grid movies and grid movie isotopies; model moves; and define grid movie moves.

To define grid movies we first need to define moves to grid diagrams corresponding to type I critical points: births, deaths and saddles (see figure \ref{fig:CriticalGridMoves}).

\begin{description}
\item[Grid birth move] A \emph{grid birth move}  takes a grid diagram $G$ to $G'$ with respective grid indices $n$ and $n+1$ by  adding a row and column and places an $O$ and an $X$ in the square that are  both in the new row and column.
\item[Grid death  move] A \emph{grid death move} is the reverse of a birth: it takes a grid diagram $G'$ to $G$ with grid indices $n+1$ and $n$ by  removing an $O$ and $X$ that are in the same row and column and deletes the row and column.
\item[Grid saddle move] A \emph{saddle move} takes a grid diagram $G$ to $G'$ with the same grid index by exchanging the places of two $O$'s such that one $O$ is directly left and down from the other $O$ to the configuration where one $O$ is directly down and right of the other.
This move cannot change the number of double points.
\end{description}
\begin{figure}[h]
  \centering
  \includegraphics[width=\textwidth]{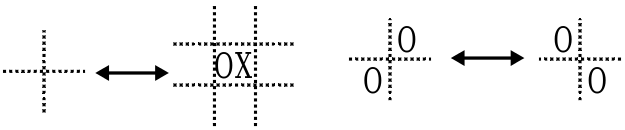}
  \caption[ \ \ Unmarked grid births, deaths and saddles]{Type I critical grid moves:birth and death on the left and the saddle move on the right.}
  \label{fig:CriticalGridMoves}
\end{figure}

\begin{defn}
  A \emph{grid movie} is a finite ordered sequence of grid diagrams such that any two sequential grid diagrams are related by a single grid move from the following list: birth, death, saddle, commutation, stabilization, destabilization and the identity.
\end{defn}

\begin{defn}
  A \emph{grid movie isotopy} is a sequence of grid movies such that the grid diagram in each level undergoes a grid planar isotopy.
\end{defn}
Note that any two grid movies can only be isotopic if they have the same number of stills.  However, if needed extra stills can be added by duplicating a still and inserting the extra copy right next to the original.

\subsection{Modified Planar Grid Algorithm and height functions in the stills}

In order to give a meaningful method of representing a smooth knotted surface diagram by a grid movie, we need to modify the planar grid algorithm to account for a height function.
We will use a complementary coordinate system $(v,v_1, v_2, v_3)$.  
Recall that $v$ is used to project to an $\R^3$ hyperplane and $v_1$ is the vector used to define the Morse function needed to get a smooth movie (without a fixed height function in each still).
We will use $v_2$ to define a generic height function in each still (perturb $v_2$ if need be), which will be used to define a smooth movie with a fixed height function in each still.  We will refer to this height function as the second height function, the $v_2$ height function, or the still height function.
Define the vector $v_r$ to be a vector that is parallel to all of the vertical lines of rectangles used in the planar grid algorithm.
Unfortunately, we cannot use $v_r=v_2$ since the horizontal lines in a grid diagram would not be generic with respect to the second height function (not to mention the common occurrence of multiple intersection points along the same horizontal line).  
Instead we will choose $v_r$ to be a vector in the plane containing $v_2,v_3$ such that the angle between $v_r$ and $v_2$ is $\theta<\arctan(\xi)$ for small $\xi>0$ (see figure \ref{fig:Dictionary}) .
We will eventually choose $\xi< \frac{1}{N}$ where $N$ is the maximum grid index of any still of the sub-collection used to represent a knotted surface diagram.
In summary, given a smooth knotted surface in 4-space we choose a complementary coordinate system to get a smooth movie with a fixed height function in each still and then, by choosing $v_r$, define the direction vertical lines are drawn in our grid diagrams.

Only two modifications need to be made to the planar grid algorithm so that it respects the height function.  
First, change step (1) of the planar grid algorithm to the following:
\begin{description}
\item[(1h)]  Draw intersection squares around each double point and draw a rectangle about each local maximum point (resp. minimum point) such that: each rectangle contains a single critical point; there are only two intersection points  occurring on the bottom and right  edges (resp. top and left edges) between a rectangle and the arcs of the diagram.  
Define these to be the critical rectangles of a grid diagram.
\end{description}
Secondly, when replacing the smooth arcs by grid arcs in step 4 of the planar grid algorithm, make sure that: (1) there is only corner in each non-intersection rectangle; and (2) the height ordering  of the horizontal grid lines of the critical rectangles have the same height ordering as the critical points they represent.

\begin{defn}
  The \emph{height planar grid algorithm} is the planar grid algorithm with the two modifications above.
\end{defn}

\subsection{Grid Movies from Smooth Movies}
We need a method of getting a grid movie from a smooth knotted surface diagram such that passing to this grid representation and back to a smooth knotted surface diagram yields a diagram isotopic to the original.  In this subsection we present such a method.  The details are in the constructive proof of the following theorem, which uses the dictionary found in  figure \ref{fig:Dictionary} between basic grid diagram moves and their corresponding sequence of FESIs.

\begin{thm}
\label{thm:MovieRepresentation}
  To any smooth knotted surface $K$ a grid movie $\G$ can be assigned.  For any grid movie $\G'$ there is a smooth knotted surface $K'$ whose corresponding grid movie is $\G'$.  Furthermore, if $\G=\G'$ then $K$ is smoothly isotopic to $K'$.
\end{thm}

\begin{proof}
  We will assume familiarity with sentences and words as used in \cite{CRS1997} for this proof.
Let $K_{[a,b]}$ be a smooth  knotted surface and let  $\K_{[a,b]}$ be a smooth movie with a height function fixed in each still of $K_{[a,b]}$.
By theorem 3.5.4 of \cite{CRS1997} there is a sentence $S$ associated to $\K_{[a,b]}$.
We will assign each particular FESI or sequence of FESIs with a sequence of grid moves as indicated by the dictionary in figure \ref{fig:Dictionary}.  In order to do this we need to perform an isotopy that will change the FESIs used to represent the surface $K_{[a,b]}$.  
Specifically, act on $\K_{[a,b]}$ with a level preserving isotopy of $\R^3$ (that does not respect the height function) such that each occurrence of the  FESIs below are sent to the sequence of FESIs indicated in figure \ref{fig:AsymmetricFESIs}.
\begin{figure}[h]
  \centering
  \includegraphics[width=.8\textwidth]{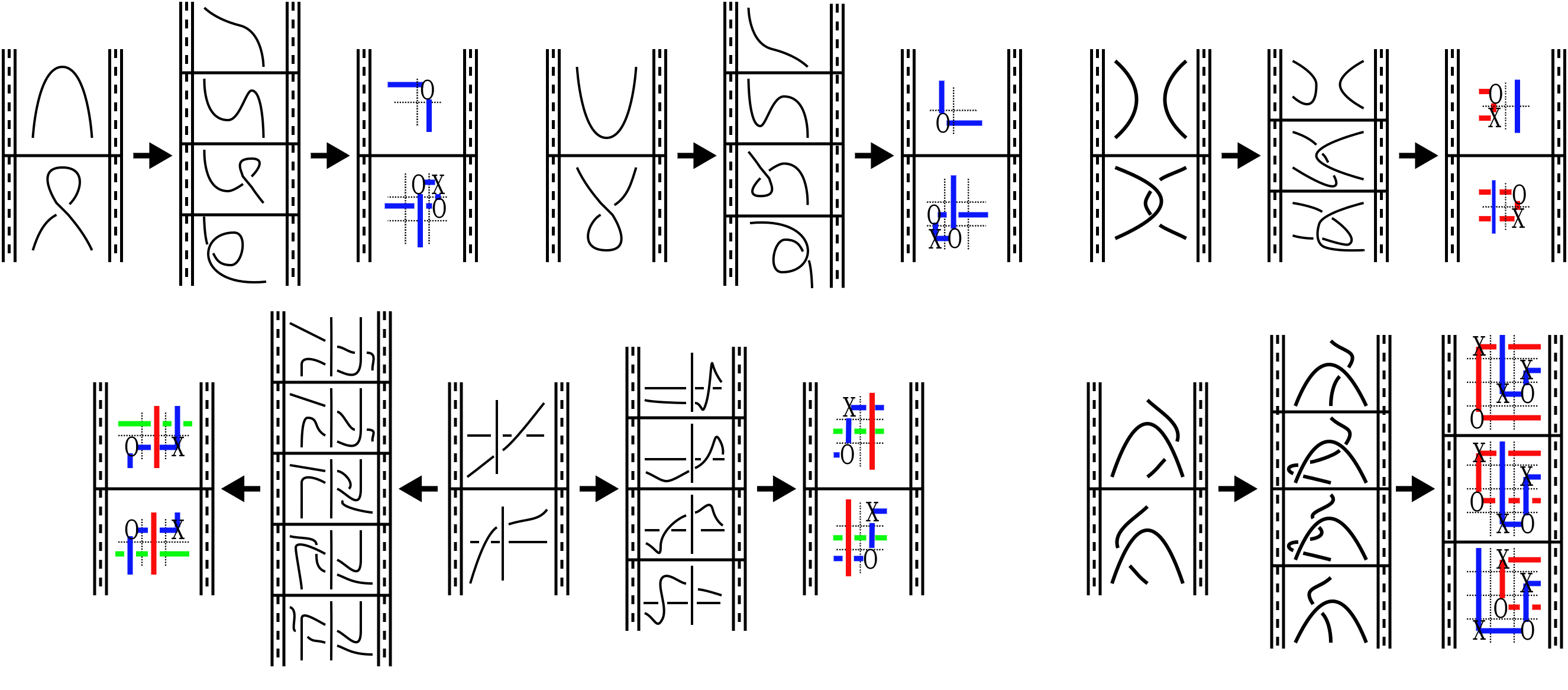}
  \caption{FESIs get mapped to a different sequence of FESIs, which will be identified with a basic grid move.   Three other R2 FESIs exist that are similar to the top right sequence as well as three other similar transfer moves.}
  \label{fig:AsymmetricFESIs}
\end{figure}

Let $\K'_{[a,b]}$ be the knotted surface diagram resulting from this isotopy and let $S'$ be the sentence associated to $\K'_{[a,b]}$. 
Except for the camel-back FESs this isotopy  only induces local planar isotopies near these type of FESIs.  In a camel-back FESI the movie move 23a is present.  Therefore, the first movie move theorem ensures that $\K_{[a,b]}$ and $\K'_{[a,b]}$ represent isotopic surfaces, which means that sentences  $S$ and $S'$ are sentences of isotopic surfaces.
To get a grid movie follow the grid movie algorithm below (let $\epsilon>0$ be small).

\textbf{Grid Movie Algorithm:}
\begin{enumerate}
 \item  Apply the height planar grid algorithm to $\K'_a\subset \R^2_a$ to get a grid diagram $G_0$ but don't delete any quadrilaterals used, call this set of quadrilaterals $Q_0$.  
If $\K'_a$ is a point then let $G_0$ be the empty grid diagram  and $G_1$ be the grid diagram  with an $X$ and $O$ in the only square and apply the algorithm to the  still $\K'_{a+\epsilon}$ to get $G_2$.
 \item Slide $Q_0$, setting $Q_s=Q_0$,   in the  time direction until the  still $\K'_{a_1}$ does not satisfy the properties set forth in the height planar grid algorithm.
 \item In the still $\K'_{a_1-\epsilon}$ modify the set $Q_0$ to the set $Q_1$, which will satisfy the needed conditions in $\K'_{a_1}$.  If this results  in more than one grid move being performed then choose an ordering for these moves and apply them individually to $G_0$ to get the next grid diagrams in the sequence $G_i, G_{i+1}, \dots, G_n$. 
 \item If a critical point, of either the Morse function or the second height function, is reached then use the dictionary in figure \ref{fig:Dictionary} to exchange specific sequences of FESIs with the necessary basic grid move representing that type of critical point.

   Specifically, for type I critical points use the corresponding grid move (birth, death, saddle).
   For type II and III critical points: a branch point corresponds to a Reidemeister Imove (achieved by  an R1-stabilization or R1- destabilization); a local minimum or maximum of the double point set (that is not an endpoint of $\overline{S_2}$) corresponds to a  Reidemeister II move (achieved by commutation);  a triple point corresponds to  a  Reidemeister III move (achieved by a commutation).  
For critical points of the second height function: a cusp move is achieved by a kink stabilization or a kink destabilization; a camel back move is achieved by a transfer move; each multi-local move is achieved by a commutation.
 \item Repeat  steps 2-4 until the still $\K'_b$ is reached.
 \item Delete all sets of rectangles $Q_l$.
\end{enumerate}

This will result in a finite sequence of grid diagrams $\G=[G_0, G_1, \ldots, G_m]$ for some $m\in \Z$ where the brackets indicate an ordered list.  
Specifically, $G_0$ corresponds to the grid diagram in stills $\R^2_{[0,a_1-\epsilon)}$ generated by quadrilaterals $Q_s$ for $s \in [0,a_1]$ and $G_1$ corresponds to the grid diagram in stills $\R^2_{[a_1-\epsilon, a_2-\epsilon)}$ generated by quadrilaterals $Q_s$ for any $s\in [a_1-\epsilon, a_2-\epsilon)$, etc.
This sequence is a grid movie $\G$ by construction and has the further property that it contains all FESIs present in sentence $S'$ along with some additional isotopy information.  Thus for any knotted surface a grid movie can be assigned to it.

\begin{figure}[p]
  \centering
  \includegraphics[width=\textwidth]{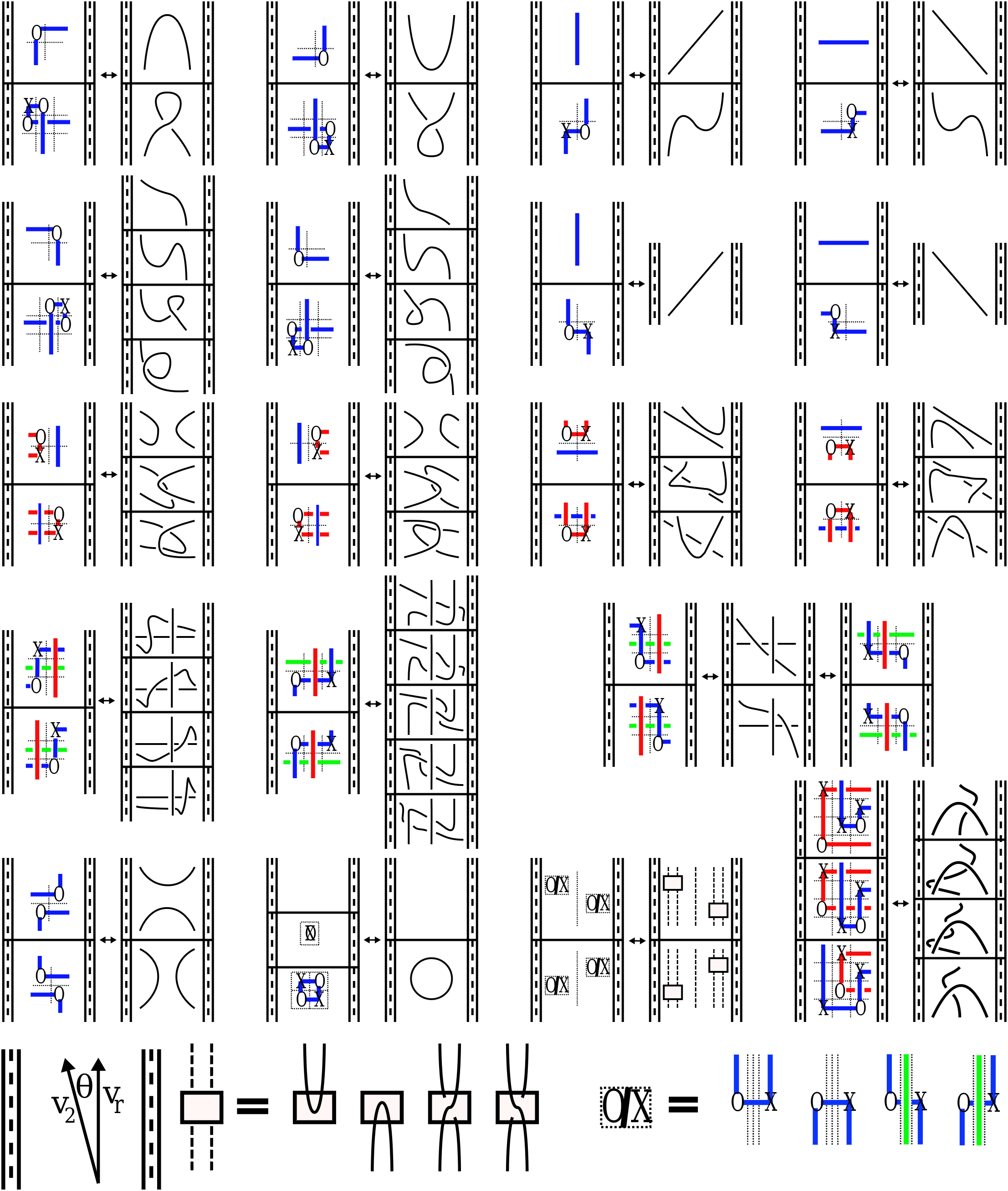}
  \caption{Dictionary of basic grid moves and their corresponding sequence of FESIs.  The 8 (de)stabilizations, four R2 commutations and the four R3 commutations  are depicted in the first and second, third, and fourth rows respectively.  The fifth row depicts the saddle, birth/death, multi-local  and camel-back moves.   }
  \label{fig:Dictionary}
\end{figure}

Next, suppose that a grid movie  $\G'=[G'_1, \ldots, G'_k]$ is given.
Use the dictionary in figure \ref{fig:Dictionary} to construct a sentence $S'$ (delete any duplicate stills to do so).
Then appeal to theorem 3.5.4 in \cite{CRS1997} to conclude that there exists a knotted surface $K''$ with movie $\K''$ whose corresponding sentence is the given one.
Next, reverse the dictionary identifications to arrive at a grid movie $\G''$.
In the process of creating a sentence from $\G'$ any extra level preserving (and height respecting) isotopy information was deleted.
Therefore, $\G''$  differs from $\G'$ by this isotopy information that does not effect the sentence $S'$. 
That is $\G''$ differs from $\G'$ by a grid movie isotopy that does not change the sentence $S'$. 
Apply a level preserving (and second height function respecting) isotopy of $\R^3$ to $\K''$ get a knotted surface movie $\K'$ representing a surface with sentence $S'$ that also has grid movie $\G'$.
Hence, there exists a knotted surface $K'$, with knotted surface diagram $\K'$, whose associated grid movie is $\G'$.

For the last part of the theorem let $K$ be a knotted surface with knotted surface diagram $\K$ and  associated grid movie $\G$ and let $K'$ be the knotted surface, with knotted surface diagram $\K'$, corresponding to $\G$ via the process described above.  
We need to show that $K$ and $K'$ are isotopic knotted surfaces.  
To see this we will simply review the procedure above and make an observation.
The knotted surface diagram  $\K$ has a sentence $S$, we then isotope this knotted surface diagram to an isotopic knotted surface diagram $\K'$ with sentence $S'$ and then associate a grid movie $\G$ to $\K'$ by the grid movie algorithm.  Clearly there exists a knotted surface  associated to $\G$ since $\K'$ gave rise to $\G$.  In fact there are many different surfaces that give rise to $\G$ and each of these have the same sentence $S'$.  Therefore, any knotted surface that gives rise to $\G$ must be isotopic to $K$ since $S$ and $S'$ represent isotopic knotted surfaces.
\end{proof}

\begin{defn}
  The grid movie $\G:  =[G_0, G_1, \ldots, G_m]$ \emph{represents} the knotted surface $K$ and the knotted surface diagram $\K=p(K)$ if $\G$ could be obtained from the process in theorem \ref{thm:MovieRepresentation}
\end{defn}

A quick aside: it is sometimes useful to think of a grid movie as a single grid diagram together with a finite ordered list of grid moves that act sequentially on the grid diagram,
\begin{equation}
  \label{eqn:sd}
  \G:  =[G_0,G_1, G_2, \ldots, G_n]=[G_0, g_1G_0, g_2g_1G_0, g_3g_2g_1G_0, \ldots \Pi_{i=1}^ng_iG_0]
\end{equation}
where $g_k$, for each $k$,  is one of the following  grid moves: birth,  death,  saddle, commutation, stabilization, destabilization, identity.

\subsection{Grid Movie Isotopies}

To get a grid movie isotopy from a smooth isotopy $\I$ of a movie $\K$ perform the grid movie algorithm to get a grid movie $\G_0$ that represents $\K_0$ but don't delete the ordered sets of rectangles $Q_l$ for $l\in [a,b]$.
Then run the isotopy forward until time $t_1$ when some still $\K_{b_1}$ with rectangles $Q_{b_1}$ doesn't satisfy the properties set forth in the planar grid algorithm.
Change the block of rectangles $Q_{[b_1-c,b_1+c]}$ (for some $c\in \R$) so that the isotopy satisfies the planar grid algorithm at $t_1$.
This will add a grid diagram to the grid movie $\G_i$ in some position $G_k$.
Call this new grid movie $\G_{i+1}$ and to make $\G_{i+1}$ grid isotopic to $\G_i$  insert a copy of $G_{k-1}$ at $G_k$ in $\G_i$ (and each grid movie $\G_n$ for $n<i$).
Repeat these steps until the isotopy is finished.

\subsection{Reidemeister, Movie, and Grid Movie Moves}
In this subsection we will discuss exactly how movie moves act on movies, develop notation that captures these changes explicitly and then apply this perspective to define grid movie moves.
To accomplish this, we will begin in the common ground of Reidemeister moves acting on a diagram of a knot.
\begin{defn}
  A \emph{sub-diagram} of a knot diagram  $E$ is a rectangular subset $F$ of $E$.
\end{defn}
Note that a sub-diagram will usually not be a diagram of a knot.  
We will choose model Reidemeister moves as shown in figure \ref{fig:ModelReidemeisterMoves}. 
\begin{figure}[h]
  \centering
  \includegraphics[width=\textwidth]{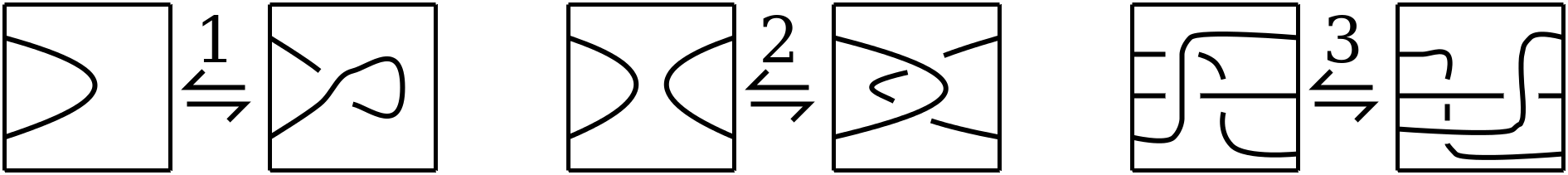}
  \caption{\ \ Model Reidemeister moves.}
  \label{fig:ModelReidemeisterMoves}
\end{figure}
We denote the left and right sub-diagrams of each Reidemeister move by $F_{iL}$ and $F_{iR}$ respectively.

To be very careful,  we will  define a function $R_i\colon E\to E'$ that operates on a knot diagram  with a sub-diagram $F_{iL}$ by  defining  $R_i$ restricted to the complement of $F_{iL}$ to be the identity and  $R_i(F_{iL})=F_{iR}$. 
Technically there are more than three moves once certain symmetries are taken into account.

Model moves are very restrictive, since there are very few diagrams that have these exact sub-diagrams.
However, we are also allowed to planar isotope a diagram so that if there is a sub-diagram planar isotopic to one of the diagrams in our model moves we are still able to perform the Reidemeister move.

Let $E$ be a diagram that has a sub-diagram that is planar isotopic (by $I$)  to a sub-diagram involved in a model Reidemeister move.  
To be very explicit, to perform a Reidemeister move on diagram $E$ the following procedure needs to be carried out: isotope $E$ by $I_s$,  to get a diagram, $I_1(E)$,  with a model Reidemeister  sub-diagram; perform the Reidemeister move $R\comp I_1(E)$; and then (if desired) reverse the isotopy $\overline{I}_t\comp R\comp I_1 (E)$, where $\overline{I}_s:  =I_{1-s}$, to arrive at a diagram $E'$.

The movie moves are also model moves that act on a movie of a knotted surface diagram.
Specifically, a movie move is a map that takes a collection of sub-diagrams $\D_{iL}:  =\cup_{v\in{[a_1,b_1]}}D_v$ (where $D_v\subset R^2_v$ is a sub-diagram)  to a different collection of sub-diagrams $\D_{iR}$.
In notation, a movie move is a map $\MM_i\colon \R^3\to \R^3$ such that $\MM_i$ is the identity on the complement of $\D_{iL}$ and $\MM_i(\D_{iL})=\D_{iR}$.
Again, since movie moves are model moves, they are quite restrictive and are combined with level preserving isotopies of $\R^3$ (that respect the $v_2$ height function) to make them robust.
Explicitly, to perform a movie move the following procedure needs to be carried out: isotope $\K_{[a,b]}$ by a level preserving isotopy $\I^t_{[a,b]}$ (that respects the $v_2$ height function), to get a movie with a model movie move; perform the movie move; and then (if desired) reverse the isotopy $\I^t$ with $\overline{\I}^t_{[a,b]}$ to arrive at a movie $\K'_{[a,b]}$.

We will define grid movie moves to be model moves as well.
This means that a grid movie isotopy  will almost always be required, since it will be rare that the exact sequence  of sub-diagrams of our model move will be contained in a grid movie.
\begin{defn}
\label{defn:SubDiagram}
  A \emph{sub-diagram}  of a grid diagram $G$  is a rectangular collection $D$ of  grid squares from $G$ that denotes the positions of $O$'s and  $X$'s  and specifies the direction of any grid arcs that intersect the boundary of $D$.
\end{defn}

From now on we will use $\D_{[c,d]}$ to indicate  a smooth collection of sub-diagrams and $\D_i$ to indicate a sequence of grid sub-diagrams.
We  denote smooth movie moves and grid movie moves with the same symbol $\MM$.
We hope the meaning will be clear from context.
We denote sub-diagrams of a grid diagram by the collection of $X$'s, $O$'s, and arcs in the sub-diagram.  



\begin{defn}
  A \emph{grid movie move} $\MM_i$ is a map
\begin{equation}
  \label{eqn:GridMovieMove}
  \MM_i\colon  \G \to \G' \ \ \ \ \ \ \mbox{ defined by }\ \ \ \ \ \ \  \begin{matrix}\MM_i(\D_{iL})= \D_{iR} \\ \ \ \MM_i|_{\D_{iL}^C}=\mbox{identity,} \end{matrix}
\end{equation}
where $\D_{iL}^C$ is the complement of $\D_{iL}$.
\end{defn}

We have defined grid movies, grid movie moves and grid isotopies.
We have also discussed how to get grid movies and grid isotopies from smooth movies and smooth isotopies.
All that remains to be done is to choose our model grid movie moves.
To do this, we will simply apply the process used in theorem \ref{thm:MovieRepresentation} to the neighborhood of a movie move.
With the following definition we have finally defined all the objects needed to prove theorem \ref{thm:GridMovieMoves}.  

\begin{defn}
  The list of  \emph{model grid movie moves} $\MM_i, i=1-30$ are in the four page figure \ref{fig:GridMovieMoves} and each is obtained by a particular application of the method presented in theorem \ref{thm:MovieRepresentation} to arcs in the stills of the smooth movie moves.  The last move $\MM_{31}$ is discussed below.
\end{defn}
\begin{figure}[p]
  \centering
  \includegraphics[width=1\textwidth]{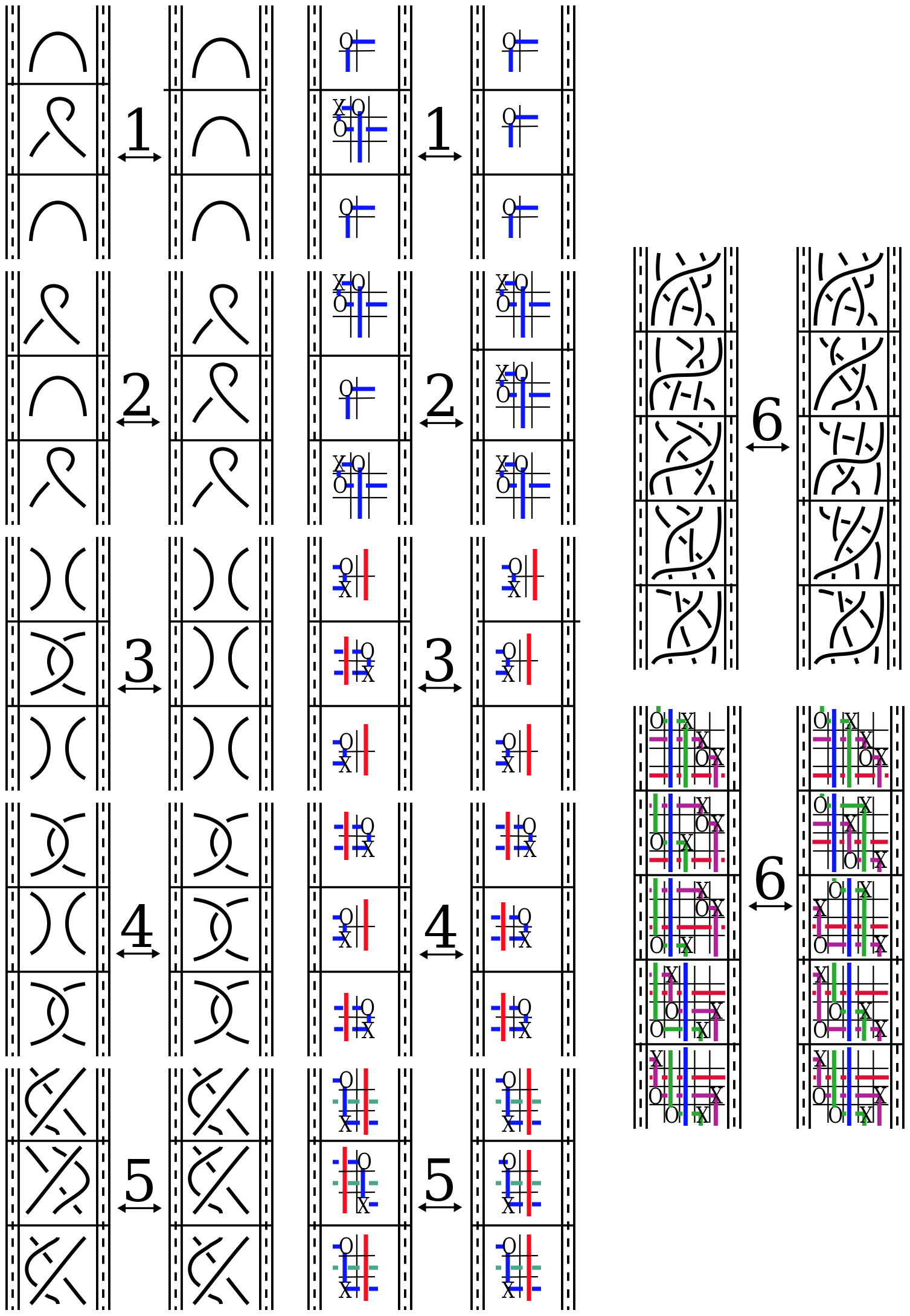}
\end{figure}
\begin{figure}[p]
  \centering
  \includegraphics[width=1\textwidth]{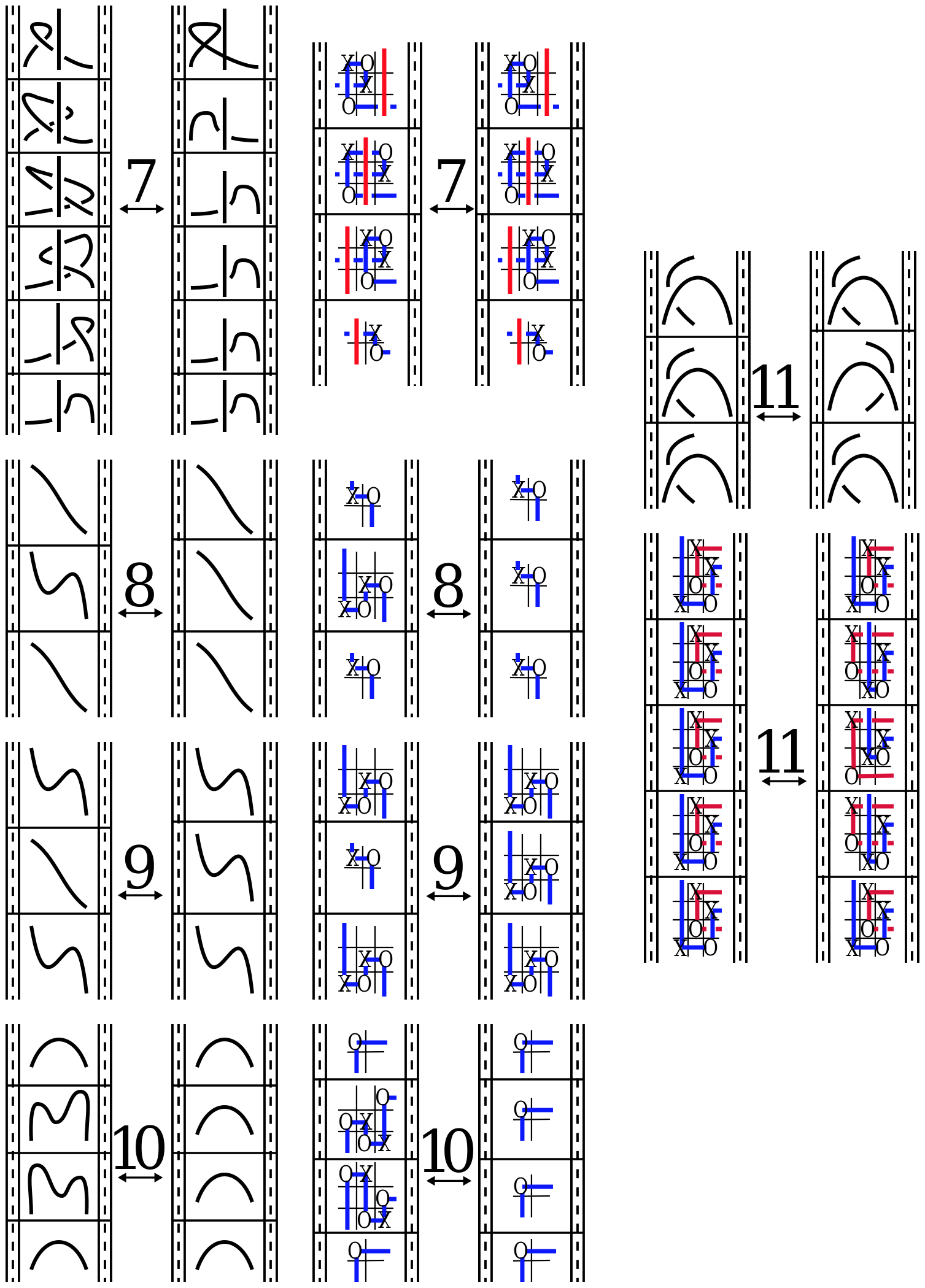}
\end{figure}
\begin{figure}[p]
  \centering
  \includegraphics[width=1\textwidth]{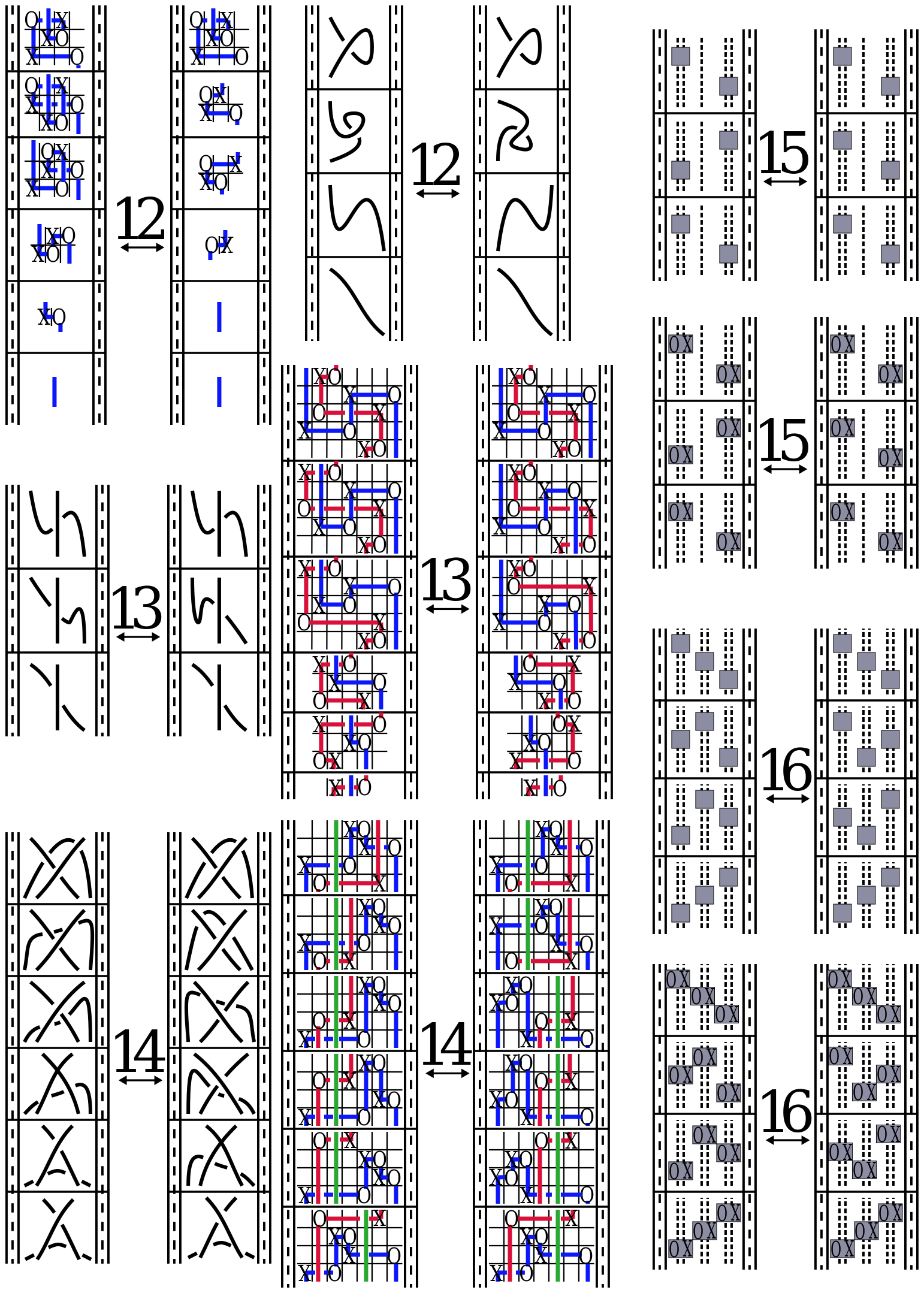}
\end{figure}
\begin{figure}[p]
  \centering
  \includegraphics[width=1\textwidth]{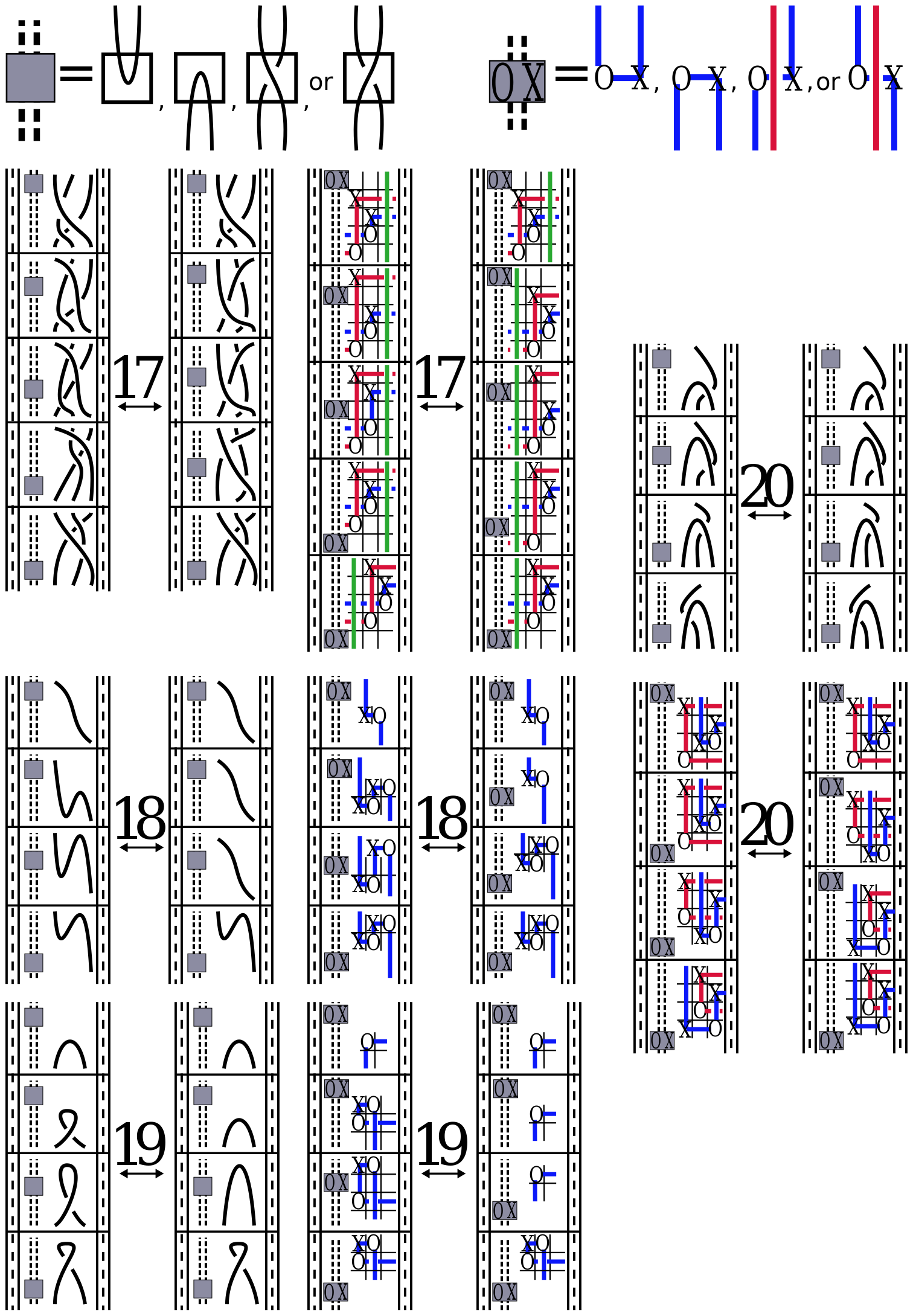}
\end{figure}
\begin{figure}[p]
  \centering
  \includegraphics[width=1\textwidth]{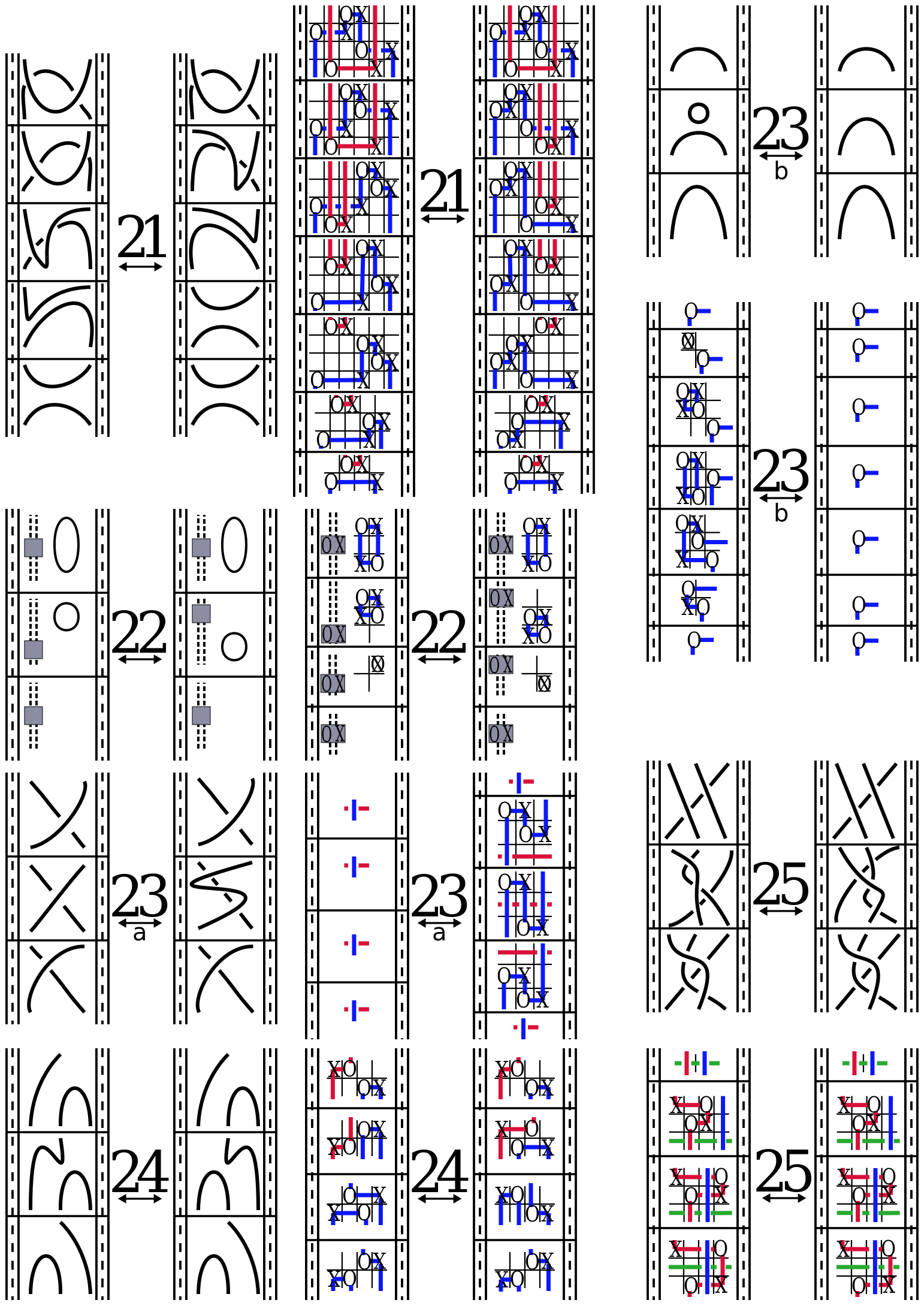}
\end{figure}
\begin{figure}[p]

   \centering
  \includegraphics[width=.95\textwidth]{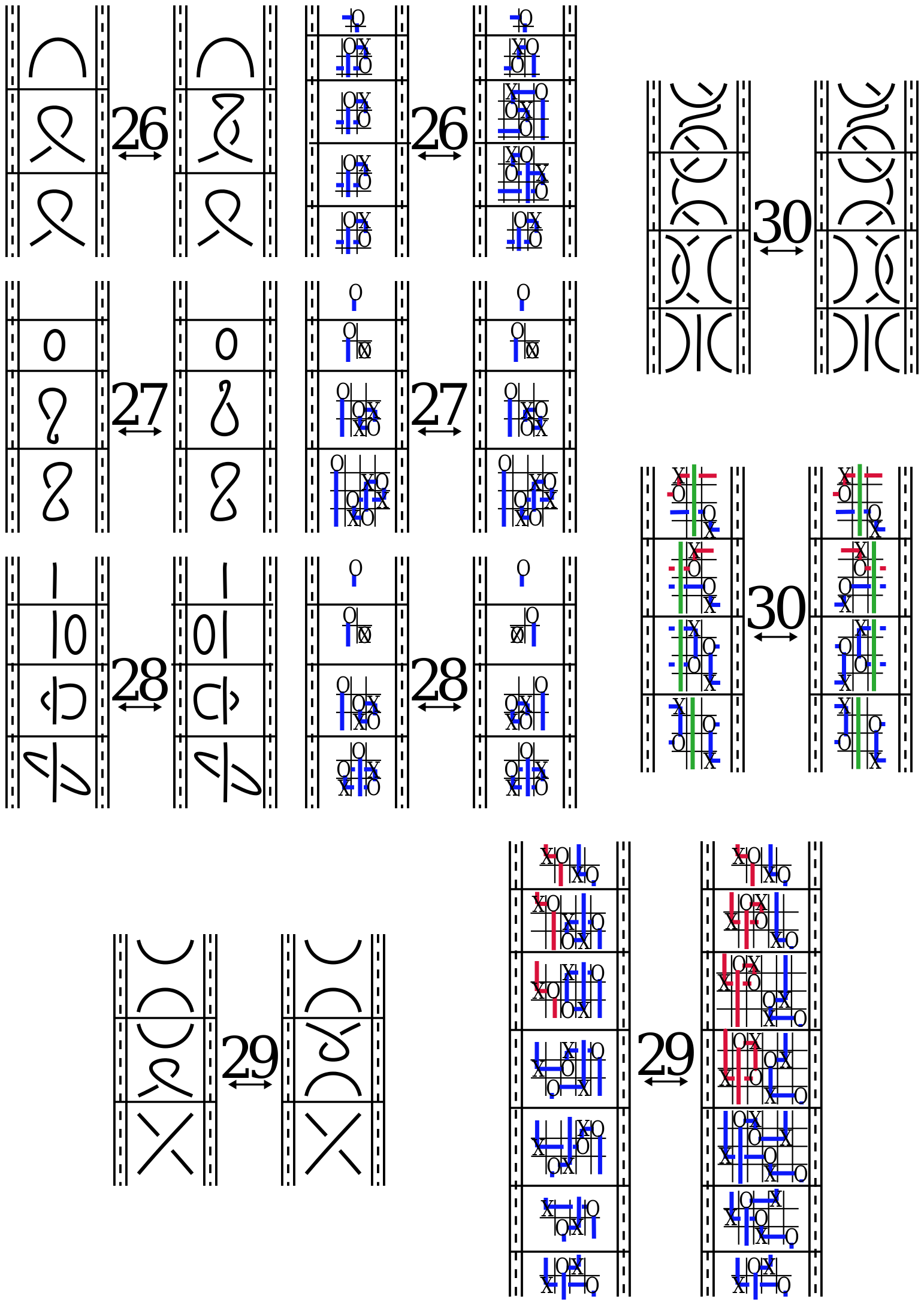}
  \caption{Smooth and Grid movie moves $1, \cdots, 30$.}
  \label{fig:GridMovieMoves}
\end{figure}

The grid diagrammatic representation of  movie move 31 can be denoted by  
\[  [F g_1 h_1E , F g_1g_2 h_1E , F g_1g_2 h_1h_2E ] \leftrightarrow [F g_1 h_1E , F g_1 h_1h_2E , F g_1g_2 h_1h_2E ],  \]
where: $E=\Pi_{i=1}^kg_iG$ is  a product of grid moves $g_i$ acting on a grid diagram $G$;   $F=\Pi_{j=m}^ng_j$ is a product of grid moves; and  $f_1, f_2, g_1, g_2$ are  grid moves that represent  FESIs.

\begin{proof} [Proof of theorem \ref{thm:GridMovieMoves}]
($\Rightarrow$)  Suppose that $\G_1,\G_2$ are grid movies that represent smooth isotopic knottings. 
This means that $\G_i$ represents a smooth knotted surface $K_i$ with knotted surface diagram $\K_i$ for $i=1,2$. 
It is given that  $K_1$ and $K_2$ are  smooth isotopic knotted surfaces, which means that we can appeal to theorem \ref{thm:SmoothMovieMoves} to get that  $\K_1$ and $\K_2$  are related by a sequence of movie moves and possibly a level preserving (and second height function respecting) isotopy. 

Since each model smooth movie move is represented by a model grid movie move (theorem \ref{thm:MovieRepresentation}), each smooth level preserving isotopy respecting the height function gives rise to a grid movie isotopy that respects the height function (by construction in section 4.3)  and the output of the  planar grid algorithm is well defined (proposition \ref{prop:WellDefinedGridAlgorithm},  theorem \ref{thm:IsotopyClasses}  and lemma \ref{lem:GridPlanarIsotopicArcs}) then $\G_1$ and $\G_2$ are related by a sequence of grid movie moves and possibly a grid movie isotopy respecting the height function.

($\Leftarrow$)  Reverse the argument given above.
\end{proof}

\section{Surfaces with boundary}
Our main motivation for this work is to investigate whether the combinatorial version of $\HFK^-$ is functorial over smooth marked cobordisms.  
The idea is to use grid movies to represent the surfaces and then compare the maps  induced on the knot Floer homology chain complex by grid movies.  
In this paper, we have presented a generalization of Carter Rieger and Saito's movie move theorem that involves closed surfaces.  
The maps induced by closed surfaces, on knot Floer homology complexes, are decidedly uninteresting, since they induce a map between chain complexes with no generators.  
The aim of this section is to generalize the definitions and constructions used previously in this paper to account for surfaces with boundary components.

Let $(F,\partial F)$ be a surface with boundary.  Choose a  vector $v_1\in \R^4$ and for each $s\in \R$ let  $\R^3_s$ denote the hyperplane orthogonal to $v_1$ at the point $sv_1$.  Again, we will let $K$ refer either to the embedding $K\colon F \hookrightarrow\R^4$ or the image $K(F)$ of the embedding and $\partial K:  =K(\partial F)$.

\begin{defn}
  An embedding of a knotted surface with boundary $(K,\partial K, v_1, a,b)$ is an embedding $K\colon F \hookrightarrow\R^4$  along with a choice of  unit vector $v_1\in \R^4$  and $a<b\in \R$ such that
  \begin{enumerate}
  \item  $\partial K\subset \R^3_a\cup \R^3_b$;
  \item $K\subset \R^3\times[a,b]$;
  \item $\left (K-\partial K\right) \subset \R^3\times (a,b)$.
  \end{enumerate}
\end{defn}
Note that the third condition  prevents type I critical points from occurring in $\R^3_a\cup \R^3_b$.
We often use $(K,\partial K)$ to denote a knotted surface with boundary when the choice of $v_1, a, b$ are not important for the discussion.
\subsection{Boundary Knotted Surface Isotopies}
 \begin{defn}
     Two knotted surfaces with boundary $(K_i, \partial K_i, v_1, a, b), \ i=1,2$  are \emph{boundary ambiently isotopic} if there is an isotopy $H\colon \R^4 \times [0,1] \to \R^4$ such that: 
     \begin{enumerate}
       \item $H(x,0)=x \ \ \forall x\in \R^4$;
       \item $H(K_1(c), 1)=K_2(c) \ \ \forall c\in F$;
       \item $H(K,s)\subset \R^3\times[a,b]  \ \forall s\in [0,1]$;
       \item $\left.H_s\right|_{\R^3_a\cup \R^3_b}= \R^3_a\cup \R^3_b \ \ \forall s\in [0,1]$.
     \end{enumerate}
   \end{defn}
   Two knotted surfaces with boundary are \emph{fixed boundary ambiently isotopic} if the first three conditions above are satisfied and the last condition is strengthened to $\left.H_s\right|_{\R^3_a\cup \R^3_b}=id$.

Note that the definition of knotted surfaces with boundary and their isotopies do not allow any part of the embedded surface to intersect, $\R^3\times (-\infty,\infty)-\R^3\times[a,b]$.  
This prevents a ``knotted'' surface with boundary from ``unknotting'' itself by pulling components over the boundary.

There are different choices of unit vector $v_1$ and $a,b \in \R$.
Therefore in order to compare embeddings with different choices of $v_1,a,b$ we need to describe isotopies that will make two choices of $v_1, a,b$ coincide.   
In order to preserve the meaning of the vectors $v_1$ and $v_2$ as the vectors corresponding to two different Morse functions we will use $w$ as our choice of unit vector for the rest of this subsection.
Suppose that $K_i\colon F_i\hookrightarrow \R^4$ for $i=1,2$ are two embeddings with choices $w_{i}, a_i,b_i$.  
Presently we define: a rotation isotopy  $R_{w_{1}w_{2}}$; a collar isotopy $C$; and a translation isotopy  $T$, which
 lead to a notion of isotopy that does not depend on the choices of $w_1,a,b$.

     \begin{defn}
       Let $w_1, w_2$ be two unit vectors in $\R^4$. A \emph{rotation isotopy} is an isotopy $R_{w_1w_2}\colon \R^4\times [0,1]\to \R^4$ that rotates $w_1$ through the plane spanned by $w_1,w_2$ such that $R_{w_1w_2}(w_1, 0)=w_1$ and $R_{w_1w_2}(w_1,1)= w_2$.
     \end{defn}
     A rotation isotopy applied to a set of points in $\R^4$ is a rigid body rotation.

     \begin{defn}
       A  \emph{collar isotopy} $C_{b,b'}$ of an embedded surface $(K, \partial K)$ is a map  $K_{[a,b]}$ to a surface $K_{[a,b']}$ for $b<b'$ defined by the function $C_{b,b'}\colon \R^3\times \R\times [0,1] \to \R^3\times \R$ satisfying:  
 \begin{align*}
   C_{b,b'}(K_t,t,s)&=\left\{\begin{matrix} (K_t,t)  &\forall t\in[a,b], \forall s\in [0,1]   \\(K_b, t) & \forall t\in [b,b'], s\in [0,1] \st t\leq b+s(b'-b)\\ 0 & \mbox{ otherwise }\end{matrix}\right.
   \end{align*}
     \end{defn}
     The resulting smooth manifold is simply the original surface with a cylinder $\partial K_b\times [b,b']$ glued to the boundary $\partial K_b$, where $\partial K_b:  =K(F)\cap\R^3_b$.

     \begin{defn}
       A \emph{translational isotopy} $T_{a,a'}$ of an embedded surface $(K, \partial K)$ is a map 
\[T_{a,a'}\colon \R^3\times \R\times [0,1]\to \R^3\times \R \mbox{ defined by }H(K_t,t,s)=(K_t, t+s(a'-a))\]
     \end{defn}
     Translational isotopies are also rigid isotopies.

     \begin{lem}
       Any two embedded surfaces with boundary $(K_i,\partial K_i,w_{i},a_i,b_i)$ can be isotoped to embedded surfaces with boundary contained in $\R^3\times[a_2,b_2]$ (decomposed with the vector $w_{2}$).
       \label{lem:RigidIsotopy}
     \end{lem}
     \proof
     Let  $(K_i,\partial K_i,w_{i},a_i,b_i)$ be two embedded surfaces with boundary.  If $i=2$ the second embedded surface already satisfies the requirement.  So it only remains to show that we can isotope the first surface to meet the requirements. Consider the following.
     \begin{align*}
       T_{a_1,a_2}C_{b_1, b_2-a_2+a_1} R_{w_{1}w_{2}}(K_1,\partial K_1, w_{1},a_1,b_1)&= T_{a_1,a_2}C_{b_1, b_2-a_2+a_1}(K_1,\partial K_1, w_{2},a_1,b_1)
       \\
       &=T_{a_1,a_2}(K_1,\partial K_1, w_{2},a_1,b_2-a_2+a_1)
       \\
       &=(K_1,\partial K_1, w_{2},a_2,b_2)   
     \end{align*}
We have abused notation and let $K_1, \partial K_1$ represent the final image of each of the isotopies. 
 \qed

   \begin{defn}
     Two embedded surfaces with boundary $(K_i,\partial K_i, w_{1},a_1,b_1)$ with $i=1,2$ are \emph{boundary isotopic} if there exists a composition of rotational, collar, translational and ambient isotopies, $D\colon \R^4\times[0,1]\to \R^4$ satisfying:
     \begin{enumerate}
     \item $D (K_1,0)=K_1$ and $D(\partial K_1,0)=\partial K_1$; 
     \item  $D (K_1,1)=K_2$ and  $D (\partial K_1, 1)=\partial K_2$.
     \end{enumerate}
   \end{defn}

\subsection{Knotted Surface Diagrams with Boundary}
   To get a knotted surface diagram of an embedded surface with boundary  $(K,\partial K, v_1,a,b)$, choose a vector $v \in v_1^\perp$ and project to some hyperplane $\R^3$ of $\R^4$.  
Next, in order to get a movie of a knotted surface diagram with boundary we simply need to check that  $p_1\colon p(K)\to \R$, defined by projection onto $v_1$, gives a generic Morse function for the knotting.
If it is then we are done.  
However, since we chose $v_1$ before we chose $v$ we may need to perturb $v_1$  to get the required Morse function $p_1$.  
To ensure that we are still working with an embedded surface with boundary we  need to perturb the boundary components of $K$ to lie in $\R^3_a$ with respect to $v_1'$ instead $v_1$.  
Furthermore, to get a complementary coordinate system, we must perturb $v_1$ so that $v_1'$ remains in $ v^\perp$.  
Since Morse functions are generic there are no obstructions to perturbing $v_1$ to  $v_1'\in v^\perp$ and since both $v_1$ and $v_1'$ are in $v^\perp$  perturbing the boundary components does not effect the boundary knotted surface diagram.
Finally, if a movie with a fixed height function in each still is desired, choose a generic vector $v_2\in span(v,v_1)^\perp$ that will be used to define the Morse height function in each still.

   \begin{defn}
     An exceptional type II critical point is a boundary point of the double point set that is not a branch point. 
   \end{defn}
   From the definition of a knotted surface with boundary, exceptional type II critical points can only occur in the projection of the  boundary (of the knotted surface with boundary).  

   \begin{defn}
     A \emph{knotted surface movie with boundary} is a  knotted surface diagram with boundary together with a Morse function that separates all type I, II, III critical points save for the exceptional type II critical points.
   \end{defn}
   A \emph{knotted surface movie with boundary and a height function in each still} is simply a  knotted surface movie with boundary along with a second Morse function defined by projecting onto $v_2$.

Note that it is possible for a boundary knotted surface movie to have a critical point contained in the boundary that is not an exceptional critical point.  However, this point cannot be a type I critical point.

We would like to characterize arbitrary boundary isotopies of knotted surfaces with boundary.
This requires us to enlarge the movie moves by adding  boundary movie moves that will account for the changes in the boundary induced by boundary isotopies.  Fortunately, the needed boundary moves are a subset of the FESIs.
\begin{defn}
  The boundary movie moves are comprised of FESIs that do not represent type I critical points (births, saddles and deaths).
\end{defn}

Let $(K,\partial K)$ be a knotted surface with boundary that is represented by the  knotted surface movie with boundary $(\K,\partial\K)$.
Given the definition of a knotted surface with boundary, type I critical points cannot occur in $\partial \K$.  
However, every other type of critical point can occur in the boundary, even though this is not the generic situation.
The types of boundary isotopies that induce a non-exceptional critical point in $\partial \K$ during the isotopy are cataloged by the FESIs that do not represent type I critical points (births, deaths, and saddles).  Specifically, an arbitrary boundary isotopy can induce the following singularities in a boundary component of $\partial \K$:
\begin{itemize}
  \item an addition or deletion of a branch point of $(\K,\partial\K)$ that induces a  Reidemeister move I  in $\partial\K$;
  \item a birth or death of a double point line in the movie that induces a Reidemeister move II in $\partial\K$; 

  \item an addition or deletion of a triple point that induces a Reidemeister move III in $\partial\K$

  \item a singularity of the second height function that induces a multi-local move in $\partial\K$. 
\end{itemize}

   \begin{thm}
       Two  knotted surface movies with boundary, with a height function fixed in each still, represent boundary isotopic knottings if and only if they are related by a finite sequence of movie moves and/or boundary movie moves.
       \label{thm:BoundaryMoviesTheorem}
   \end{thm}
\proof
 Let  $(K_i,\partial K_i, v_{i},a_i,b_i)$ for $i=1,2$ be two boundary isotopic  knotted surfaces with boundary.
Using lemma \ref{lem:RigidIsotopy} we can isotope these surfaces to  $(K_i,\partial K_i, v_1,a,b)$ for some $v_1, a,b$.  
We need to  show that $(\K_i,\partial \K_i, v_1, v_2, a, b)$ for $i=1,2$,  the  movies (with fixed height functions in each still)  representing these knotted surfaces with boundary, are related by a finite sequence of movie moves and boundary movie moves iff  $(K_1, \partial K_1)$ and $(K_2, \partial K_2)$ are boundary ambiently isotopic.  Further suppose that $p\colon \R^4\to \R^3$ is the generic projection used in the movie description.

First assume that $(K_i, \partial K_i)$ for $i=1,2$ are related by a boundary isotopy $I\colon \R^4\times [0,1]\to \R^4$ and assign sentences $S_1, S_2$ to the movies $(\K_1, \partial \K_1), (\K_1, \partial \K_1)$.
The boundary isotopy can either induce singularities, of the projection, in the interior or the boundary of  $p\comp I(\K,s)$.
By appealing to the enumeration of singularities found in the proof of the smooth movie theorem \cite{CRS1997} we know that if a singularity  occurs in the interior of $p\comp I(\K_1,s)$, for some $s\in [0,1]$, then the isotopy will induce a change in the sentence $S_1$ that corresponds to a  movie move.
If the  singularity  occurs on the boundary, following the enumeration of possible boundary singularities above,  this will induce a change in the sentence corresponding to one of the boundary movie moves.  
Hence, there exists a sequence of movie moves and boundary movie moves that relate $S_1$ to $S_2$ (and therefore $(\K_1,\partial \K_1)$ to  $(\K_2, \partial \K_2)$).

The fact that for each  movie move and boundary movie move there exists an isotopy of the movie that induces the required  move allows us to reverse the proof.
\qed

If desired one could work with a definition of boundary isotopy that required the isotopy of the boundary be a planar isotopy.  
\begin{defn}
   Planar boundary movie moves consist of the multi-local move FESIs: all FESIs that do not represent type I, II or III critical points.
\end{defn}

\begin{cor}
  \label{cor:PlanarBoundaryMovieTheorem}
       Two  knotted surfaces with boundary $(K_i,\partial K_i)$ for $ i=1,2$ with movies (with a height function fixed in each still)  $(\K_i,\partial \K_i)$  are related by a boundary isotopy that induces a planar isotopy in $\partial \K_i$  if and only if they are related by a finite sequence of movie moves and planar boundary movie moves.  
\end{cor}

\begin{proof}
  The proof follows immediately from the theorem by noting that since the boundary isotopy must induce a planar isotopy in the boundary of the movie, no critical points of type I, II or III can occur in the boundary.  This only leaves the possibility of critical points of the second height function occurring in the boundary, which induce the planar boundary movie moves defined above.
\end{proof}

Finally, one could consider boundary isotopies that leave the boundary fixed.
\begin{cor}
  \label{cor:FixedBoundaryMovieTheorem}
       Two boundary knotted surface movies (with a height function fixed in each still)  are related by a boundary isotopy that fixes the boundary  if and only if they are related by a finite sequence of movie moves.    
\end{cor}
\begin{proof}
  This follows from the main theorem by observing that if the isotopy fixes the boundary then no singularities, except for exceptional type II singularities will occur in the projection of the boundary.  
This means that all  singularities induced by the isotopy must occur in the interior of the movie.  
Each of these singularities give rise to a movie move.
\end{proof}

\begin{defn}
  A grid movie with boundary, $(\G, \partial \G)$, represents the knotted surface with boundary $(K,\partial K)$ and the knotted surface diagram with boundary $(\K, \partial \K)$ if $(\G, \partial \G)$ could be obtained from $(\K, \partial \K)$ from the process in theorem \ref{thm:MovieRepresentation}. 
\end{defn}

\begin{defn}
  Boundary grid movie moves consist of sequences of grid diagrams found in the dictionary (figure \ref{fig:Dictionary}) that do not correspond to a single smooth still and that do not represent  type I critical points.
\end{defn}
\begin{thm}
\label{thm:GridMovieMoveTheoremWithBoundary}
    Two grid movies with boundary represent isotopic knotted surfaces with boundary if and only if they are related by  a sequence of grid movie moves and/or boundary grid movie moves.

\end{thm}
\begin{proof}
Let $(\G_1, \partial \G_1)$ and $(\G_2, \partial \G_2)$ represent isotopic knottings with boundary $(K_i, \partial K_i)$ with movies $(\K_i, \partial \K_i, v_1,v_2,a,b)$ for $i=1,2$ (that have a fixed height function in each still).
Since $(K_1,\partial K_1)$ is boundary isotopic to $(K_2,\partial K_2)$ theorem \ref{thm:BoundaryMoviesTheorem} implies that there exists a finite sequence of movie moves and/or boundary movie moves relating  $(\K_1, \partial \K_1, v_1,v_2,a,b)$ and  $(\K_2, \partial \K_2, v_1,v_2,a,b)$.
For each one of these movie moves and boundary movie moves there is a corresponding grid movie move or boundary grid movie move gotten from the dictionary in figure \ref{fig:Dictionary}.  
Therefore,  $(\G_1, \partial \G_1)$ and $(\G_2, \partial \G_2)$  are related by a finite sequence of grid movie moves and/or boundary grid movie moves.
Reverse the argument to get the other direction.

\end{proof}

\begin{defn}
  Planar boundary grid movie moves consist of sequences of grid diagrams found in the dictionary (figure \ref{fig:Dictionary}) that do not correspond to a single smooth still and that do not represent  type I, II or III critical points.
\end{defn}
Note that transfer moves represent camel-back moves, and therefore are a planar boundary grid move,  even though they contain a birth and death of a double point arc.

\begin{cor}
    Two grid movies with boundary represent two  knottings with boundary $(K_i, \partial K_i), i=1,2$ with movies $(\K_i, \partial \K_i, v_1, v_2, a, b)$ that are related by a boundary isotopy that induces a planar isotopy in $\partial \K_i$, if and only if they are related by  a sequence of grid movie moves and/or boundary grid movie moves.
\end{cor}
\begin{proof}
This follows from  theorem \ref{thm:GridMovieMoveTheoremWithBoundary} and corollary \ref{cor:PlanarBoundaryMovieTheorem}.
\end{proof}
\begin{cor}
  Two grid movies with fixed boundary represent isotopic knotted surfaces with fixed boundary if and only if they are related by a sequence of movie moves.
\end{cor}

\begin{proof}
This follows from theorem \ref{thm:GridMovieMoveTheoremWithBoundary} and corollary \ref{cor:FixedBoundaryMovieTheorem}.
\end{proof}

\bibliographystyle{amsplain}
\bibliography{Research}
\end{document}